\documentclass[oneside,11pt]{amsart}
\usepackage{amsmath, amsfonts,amsthm,times,graphics}
\usepackage{mathrsfs}
\usepackage[usenames]{color}
\usepackage[active]{srcltx}
 \makeatletter
\renewcommand*\subjclass[2][2000]{%
  \def\@subjclass{#2}%
  \@ifundefined{subjclassname@#1}{%
    \ClassWarning{\@classname}{Unknown edition (#1) of Mathematics
      Subject Classification; using '1991'.}%
  }{%
    \@xp\let\@xp\subjclassname\csname subjclassname@#1\endcsname
  }%
}
 \makeatother

\newtheorem*{ThmA}{Theorem A}
\newtheorem*{ThmB}{Theorem B}
\newtheorem*{ThmC}{Theorem C}
\newtheorem*{ThmD}{Theorem D}
\newtheorem*{ThmE}{Theorem E}
\newtheorem*{ThmF}{Theorem F}
\newtheorem*{ThmG}{Theorem G}
\newtheorem*{ThmH}{Theorem H}
\newtheorem*{ThmI}{Theorem I}
\newtheorem*{ThmJ}{Theorem J}
\newtheorem*{ThmK}{Theorem K}
\newtheorem*{ThmL}{Theorem L}
\newtheorem*{ThmM}{Theorem M}

\newtheorem{Thm}{Theorem}[section]
\newtheorem{Cor}[Thm]{Corollary}

\newtheorem{Pro}[Thm]{Proposition}
\theoremstyle{definition}
\newtheorem{Def}[Thm]{Definition}

\theoremstyle{remark}
\newtheorem{Rem}[Thm]{\upshape\bfseries Remark}

\numberwithin{equation}{section}

\newcommand{\ee}{\mathrm{e}}

\theoremstyle{definition}

\def\be{\begin{equation}}
\def\ee{\end{equation}}

\newcommand{\ben}{\begin{enumerate}}
\newcommand{\een}{\end{enumerate}}

\newcommand{\B}{{\mathbb B}}
\newcommand{\C}{{\mathbb C}}

\newcommand{\U}{{\mathbb U}}

\newcommand{\br}{\begin{rem}}
\newcommand{\er}{\end{rem}}
\newcommand{\brs}{\begin{rems}}
\newcommand{\ers}{\end{rems}}
\newcommand{\bo}{\begin{obser}}
\newcommand{\eo}{\end{obser}}
\newcommand{\bos}{\begin{obsers}}
\newcommand{\eos}{\end{obsers}}
\newcommand{\bpf}{\begin{pf}}
\newcommand{\epf}{\end{pf}}
\newcommand{\ba}{\begin{array}}
\newcommand{\ea}{\end{array}}
\newcommand{\beq}{\begin{eqnarray}}
\newcommand{\beqq}{\begin{eqnarray*}}
\newcommand{\eeq}{\end{eqnarray}}
\newcommand{\eeqq}{\end{eqnarray*}}

\numberwithin{equation}{section}

\newcounter{minutes}
\divide\time by 60
\newcounter{hours}
\multiply\time by 60 \addtocounter{minutes}{-\time}
\begin{document}
\title{Schwarz type lemmas and their applications in Banach spaces}

\author[Shaolin Chen, Hidetaka Hamada, Saminathan Ponnusamy, and Ramakrishnan Vijayakumar]{Shaolin Chen, Hidetaka Hamada, Saminathan Ponnusamy, and Ramakrishnan Vijayakumar}

\address{S. L.  Chen, College of Mathematics and
Statistics, Hengyang Normal University, Hengyang, Hunan 421002,
People's Republic of China; Hunan Provincial Key Laboratory of Intelligent Information Processing and Application,  421002,
People's Republic of China.} \email{mathechen@126.com}

\address{H. Hamada, Faculty of Science and Engineering, Kyushu Sangyo University,
3-1 Matsukadai 2-Chome, Higashi-ku, Fukuoka 813-8503, Japan.}
\email{ h.hamada@ip.kyusan-u.ac.jp}

\address{S. Ponnusamy, Department of Mathematics, Indian Institute of Technology Madras,
Chennai-600 036, India.} \email{
samy@iitm.ac.in}

\address{R. Vijayakumar,
Department of Mathematics, Indian Institute of Technology Madras,
Chennai-600 036, India. } \email{mathesvijay8@gmail.com}


\maketitle

\def\thefootnote{}
\makeatletter\def\thefootnote{\@arabic\c@footnote}\makeatother

\begin{abstract}
The main purpose of this paper is to develop some methods to  investigate  the Schwarz type lemmas of holomorphic mappings and pluriharmonic
mappings in  Banach spaces. Initially, we  extend the classical Schwarz lemmas of holomorphic mappings to Banach spaces, and then we apply these extensions  to establish a sharp Bloch type theorem for pluriharmonic mappings on homogeneous unit balls of $\C^n$ and to obtain some sharp  boundary Schwarz type lemmas for holomorphic mappings in Banach spaces.
Furthermore, we improve and generalize the  classical Schwarz lemmas of planar harmonic mappings  into the sharp forms of Banach spaces, and present some applications to sharp boundary Schwarz type lemmas for pluriharmonic mappings in Banach spaces. Additionally, using a relatively simple method of proof,  we prove some sharp Schwarz-Pick type estimates of pluriharmonic mappings  in JB$^*$-triples, and the obtained results provide the improvements and generalizations
 of the corresponding  results
in \cite{CH20}.
\end{abstract}

\maketitle
\tableofcontents

\section{Preliminaries }\label{Intro}
It is well known that the Schwarz lemma has become a crucial theme in lots of branches of mathematical research for more
than a hundred years to date. We refer the reader to \cite{Ah,BK,CLW,EJLS,E-1,E-2,GHK-JAM,HRS,Lelong89,Liu-21,L-T,Ro,Wu,Y} for more details on this topic.
This paper continues the study of the classical Schwarz lemmas of
holomorphic mappings and harmonic mappings (or complex-valued harmonic functions).
First, we  extend the classical Schwarz lemmas of holomorphic mappings to   Banach spaces, and then we use the obtained
results to establish a sharp Bloch type theorem for pluriharmonic mappings on homogeneous unit balls of $\C^n$
and  obtain sharp  boundary Schwarz type lemmas for holomorphic mappings in Banach spaces. In addition, we improve and generalize the
classical Schwarz lemmas of planar harmonic mappings into the sharp forms of Banach spaces, and obtain some applications to sharp  boundary Schwarz
type lemmas for pluriharmonic mappings in Banach spaces. At last, we use a relatively simple method to
 prove some sharp Schwarz-Pick type estimates of pluriharmonic mappings in JB$^*$-triples, and the obtained results are also the improvements and generalizations  of the corresponding known results.

In order to state our main results, we need to recall some basic definitions and introduce some necessary terminologies.

Let $\mathbb{C}^{n}$ be the  complex space of dimension $n\ (n \geq 1)$, and $\|\cdot\|_{e}$ be the Euclidean norm on $\mathbb{C}^{n}$. For real or complex Banach spaces $X$ and $Y$ with norm $\|\cdot\|_{X}$ and $\|\cdot\|_{Y}$, respectively, let $L(X,Y)$ be the space of all continuous linear operators from $X$ into $Y$ with the standard
operator norm  $$\|A\|=\sup_{x\in X\backslash\{0\}}\frac{\|A
x\|_{Y}}{\|x\|_{X}},$$ where $A\in L(X,Y)$.
Then $L(X,Y)$ is a Banach space with respect to this norm.
Denote by $X^{\ast}$ the
dual space of the real or complex Banach space $X$. For $x\in
X\backslash\{0\}$, let
$$T(x)=\{l_x\in X^*:\ l_x(x)=\|x\|_X~\mbox{and}~ \|l_x\|_{X^*}=1\}.$$
Then the well known Hahn-Banach theorem implies that $T(x)\neq\emptyset$.

Let $\psi$ be a mapping of a domain $\Omega\subset X$ into a real or complex Banach space $Y$, where
$X$ is a complex Banach space.
We say that $\psi$ is differentiable at $z\in \Omega$
if there exists a bounded real linear operator $D\psi(z):\, X\to Y$
such that
\[
\lim_{\| \tau\|_X\to 0^{+}}\frac{\|
\psi(z+\tau)-\psi(z)-D\psi(z)\tau\|_Y}{\| \tau\|_X}=0.
\]
Here $D\psi(z)$ is called the Fr\'{e}chet derivative of $\psi$ at $z$. If $Y$ is a complex Banach space and $D\psi(z)$ is bounded complex linear for each $z\in \Omega$, then $\psi$ is said to be holomorphic on $\Omega$.

Also, for $z\in \Omega\setminus \{ 0\}$,
if
\[
\lim_{r\to 1^{-}}\frac{\psi(rz)-\psi(z)}{r-1}
\]
exists, then we call this the radial derivative of $\psi$ at $z$.
If $\psi$ is differentiable at $z\in \Omega \setminus \{ 0\}$,
then the radial derivative of $\psi$ at $z$ is equal to $D\psi(z)z$.
So, in general, we denote the radial derivative of $\psi$ at $z$ by $D\psi(z)z$.

For a differentiable mapping $\psi:\, \Omega\to Y$ and for a point $z_0\in \Omega$ which satisfies one of the following conditions:
\begin{enumerate}
\item[(i)]
$\psi(z_0)=0$;
\item[(ii)]
$\psi(z_0)\neq 0$ and $\|\psi(z)\|_Y$ is differentiable at $z=z_0$,
\end{enumerate}
we define
\[
|\nabla \| \psi\|_Y(z_0)|=
\sup_{\| \beta\|_X=1}\lim_{\mathbb{R}\ni t\to 0^+}
\frac{\left| \|\psi(z_0+t\beta)\|_Y-\| \psi(z_0)\|_Y\right|}{t}.
\]
As in the proof of \cite[eq.(3.1)]{X18},
we obtain the following {result}.

\begin{Pro}\label{Po-1.0}
\begin{equation}
\label{nabla-norm-differentiable}
|\nabla \| \psi\|_Y(z_0)|=
\left\{
\begin{array}{ll}
\| D\psi(z_0)\| & \mbox{if $\psi(z_0)=0$};\\
\sup_{\| \beta\|_{X}=1}\left|l_{\psi(z_0)}(D\psi(z_0)\beta)\right| & \mbox{if $\psi(z_0)\neq 0$},
\end{array}
\right.
\end{equation}
where
$l_{\psi(z_0)}\in T(\psi(z_0))$.
\end{Pro}

Let $\Omega$ be a domain in a complex Banach space $X$. A mapping $f$ of
$\Omega$ into a real or complex Banach space $Y$ is said to be pluriharmonic if
the restriction of $l(f(\cdot))$ to every holomorphic curve is
harmonic for any $l\in Y^*$ (cf. \cite{CH20,DHK-2011,Iz,Ra, RU}). In particular, if $\Omega$ is a balanced domain,
for a pluriharmonic mapping $f:\,\Omega\to Y$ and $w\in \partial \Omega$,
then we let
\begin{eqnarray*}
{\Lambda_f(0;w)}
&=&
\sup\{ |\varphi[f,w,l_u]_{\zeta}(0)|:\, l_u\in T(u), \| u\|_Y=1\}
\\
&&\quad
+
\sup\{ |\varphi[f,w,l_u]_{\overline{\zeta}}(0)|:\, l_u\in T(u), \| u\|_Y=1\},
\end{eqnarray*}
where
\[
\varphi[f,w,l_u](\zeta)=
l_u\left(f\left(\zeta w\right)\right) ,
\quad
\zeta \in \U
\]
and
$\U$ is the open unit
disk of the complex plane $\mathbb{C}$.
We note that the inequality $\Lambda_f(0;w)\leq \frac{4}{\pi}$
always holds for pluriharmonic mappings $f:\,\Omega\to Y$ with $\| f(x)\|_Y<1$
for all $x\in \Omega$ by the harmonic Schwarz lemma on the unit disk. If $Y=\mathbb{C}^{n}$ or $Y=\ell_2$,
where
\[
\ell_2=\left\{ z=(z_1, z_2,\dots): z_j\in \mathbb{C}, \quad
\sum_{j=1}^{\infty}|z_j|^2<\infty \right\},
\]
then the mapping $f=h+\overline{g}$ is pluriharmonic on $\Omega$, where $h$ and $g$ are holomorphic in $\Omega$. In this case,
\[
\Lambda_f(0;w)=\left\| Dh(0){w}\right\|_Y
+
\left\| \overline{Dg(0)w}\right\|_Y,
\quad w\in \partial \Omega.
\]

Furthermore, if $X=Y=\mathbb{C}^{n}$  and $\Omega$ is a simply connected
domain, then  $f:\,\Omega\to\mathbb{C}^{n}$ is pluriharmonic if and only
if it has a representation $f=h+\overline{g}$, where $h$ and $g$ are
holomorphic in $\Omega$. This representation is unique if
$g(0)=0$ (cf. \cite{DHK-2011,Ru2,Vl}). For a pluriharmonic mapping
$f=h+\overline{g}:\, \Omega\to\mathbb{C}^{n}$, it is an elementary exercise
to see that the real Jacobian  determinant of $f$ can be written as


\begin{equation}\nonumber
\det J_{f}=\det\left[
\begin{array}{cc}  
Dh& \overline{Dg} \\
Dg & \overline{Dh}
\end{array}
\right ]                               
\end{equation}
and  if $h$ is locally biholomorphic  in $\Omega$, then the determinant
of $J_{f}$ has the form

\beqq\label{eq-ex2a} \det J_{f}=|\det
Dh|^{2}\det\left(I-Dg[Dh]^{-1}\overline{Dg[Dh]^{-1}}\right), \eeqq
where $I$ is the identity operator (see \cite{DHK-2011}).

For an $n\times m$ complex matrix $A=(a_{ij})$,
the Frobenius norm of $A$ is defined as follows:
\[
\| A\|_F=
\sqrt{\sum_{i=1}^n\sum_{j=1}^m |a_{ij}|^2}.
\]
Then we have
\begin{equation}
\label{A-Frobenius-Operator}
\| A\|_F^2
\leq
m\| A\|^2,
\end{equation}
where
{\[
\|A\|=\sup_{\xi\in \C^m\backslash\{0\}}\frac{\|A
\xi\|_{e}}{\|\xi\|_{e}}.
\]}
Let $\Omega$ be a domain in $\C^m$.
For a pluriharmonic mapping $f:\, \Omega\to \C^n$,
let
\[
\nabla f(z)
=
\left(
\frac{\partial f}{\partial x_1}(z),
\frac{\partial f}{\partial y_1}(z),
\dots,
\frac{\partial f}{\partial x_m}(z),
\frac{\partial f}{\partial y_m}(z)
\right),
\]
where
$z=(z_1,\dots, z_m)\in \Omega$ and
$z_j=x_j+iy_j$ for $j=1, \dots, m$.

\begin{Def}
\label{d.triple} A complex Banach space $X$ is called a {\it
JB$^*$-triple} if there exists a triple product
$\{\cdot,\cdot,\cdot\} :\, X^3 \to X$ which is conjugate linear in the
middle variable, but linear and symmetric in the other variables,
and satisfies
\begin{enumerate}
\item[(i)] $\{a,b,\{x,y,z\}\} = \{\{a,b,x\},y,z\}\,-\, \{x,
\{b,a,y\}, z\}\, +\, \{x,y,\{a,b,z\}\}$;
\item [(ii)] the map
$a\square a:\,  x\in X\mapsto \{a,a,x\} \in X$ is hermitian with
nonnegative spectrum;
\item [(iii)] $\|\{a,a,a\}\|_X = \|a\|_X^3;$
\end{enumerate}
for $a,b,x,y, z \in X$.
\end{Def}

Let  $\Omega$ be a domain in a complex Banach space $X$. Denote by $\mbox{Aut}(\Omega)$ the set of biholomorphic automorphisms of $\Omega$.
A domain $\Omega\subset X$  is said to be homogeneous if for each $x,y\in\Omega$, there exists some mapping $f\in\mbox{Aut}(\Omega)$ such that
$f(x)=y$. It is known that every bounded symmetric domain in
a complex Banach space $X$ is homogeneous. Conversely, the open unit ball $B_{X}$  of $X$ admits a symmetry $s(z)=-z$ at $0$ and if $B_{X}$
is homogeneous, then $B_{X}$ is a symmetric domain. It is well known that the Euclidean unit ball in $\mathbb{C}^{n}$, the polydisc in  $\mathbb{C}^{n}$ and the classical Cartan domains
are bounded symmetric domains in $\mathbb{C}^{n}$.  Banach spaces with a homogeneous open unit ball are precisely the JB$^*$-triples.
 We refer to \cite{C12,C-2021,CHHK17,Mo-1,Mo-2, Mo-3} for more
details of JB$^*$-triples and bounded symmetric domains.

Let $X$ be a JB$^*$-triple. For every $z,w\in X$, the Bergman
operator $B(z,w)\in {L}(X)$ is defined by
$$
B(z,w)(\cdot)={I}-2z\square w+\{z,\{w,\cdot,w\},z\},
$$
where  $z\square w(x)=\{z,w,x\}$. Let $\B_X$ be the unit ball of
$X$. Then, for each $a \in \B_X$, the M\"{o}bius transformation
$g_a$ defined by
\begin{equation}
\label{Mobius} g_a(z)=a+B(a,a)^{1/2}(I+z\square a)^{-1}z,
\end{equation}
is a biholomorphic automorphism of $\B_X$ with $g_a(0)=a$,
$g_a(-a)=0$, $g_{-a}=g_a^{-1}$ and $Dg_a(0)=B(a,a)^{1/2}$. By
\cite[Corollary 3.6]{K94}, we have
\begin{equation}
\label{Kaup-estimate} \| Dg_{z_0}(0)^{-1}\| =\|
Dg_{-z_0}(z_0)\|=\frac{1}{1-\| z_0\|_X^2}.
\end{equation}

Given JB*-triples $X_1$, \ldots, $X_n$, we can form the
$\ell^\infty$-sum $X= X_1 \oplus \cdots \oplus X_n$ which becomes a
JB*-triple equipped with the coordinatewise triple product:
\[
\{ x,y,z\}=(\{x_1,{y_1},z_1\}, \cdots, \{x_i,{y_i},z_i\}, \cdots,
\{x_n,{ y_n},z_n\})\] for $x=(x_i),\, y=(y_i),\, z=(z_i)\in X$. Let
$\B_{X_j}$ be the open unit ball of $X_j$ for $j=1,\dots, n$. Then
their product $\B_{X_1}\times \cdots\times \B_{X_n}$ is the open
unit ball of $\B_X$ of the JB$^*$-triple $X$. Let $g_{j,a_j}$
($a_j\in \B_{X_j}$) be the M\"{o}bius transformation of $\B_{X_j}$
for $j=1, \dots, n$. Then, for $a=(a_1, \dots, a_n)\in \B_X$,
\[
g_a(z)=(g_{1,a_1}(z_1), \dots, g_{n,a_n}(z_n)), \quad z=(z_1, \dots,
z_n)\in\B_X,
\]
is the M\"{o}bius transformation of $\B_{X}$.

\section{Schwarz  type lemmas of holomorphic mappings and their applications}\label{sec2}

The classical Schwarz lemma
states that every holomorphic mapping $f$ of the unit disk $\U$ into itself
 with $f(0)=0$ satisfies $|f(z)|\leq|z|$ for all $z\in\U.$ Moreover, unless $f$ is a rotation, one has the strict inequality $|f'(0)|<1$, and $f$ maps each disk
 $\U_{r}:=\{z:\, |z|< r<1\}$ into a strictly smaller one.
Lindel\"of removed the assumption ``origin is a fixed point" and
improved the classical Schwarz lemma of holomorphic mappings into
the following form.

 \begin{ThmA}\label{Thm-CPV-1}{\rm  (\cite[Proposition  2.2.2]{Kra})}
Let $f$ be a holomorphic mapping  of $\mathbb{U}$ into itself. Then,
for $z\in\mathbb{U}$, the following sharp estimate
\be\label{Sch-1-1} |f(z)|\leq\frac{|f(0)|+|z|}{1+|f(0)|\,|z|} \ee
holds.
\end{ThmA}

\begin{Rem}\label{cor-Kr-1}
	By the maximum modulus principle, Theorem A continuous to hold for holomorphic  mapping $f$ from $\mathbb{U}$ into $\overline{\mathbb{U}}$.
\end{Rem}

Under the assumption of Theorem A, if   ``$|f(z)|$" in (\ref{Sch-1-1}) is replaced by ``$|f(z)-f(0)|$", then Harris \cite{H-1977} obtained
the following sharp estimate
\be\label{thk-1}|f(z)-f(0)|\leq|z|\frac{1-|f(0)|^{2}}{1-|f(0)||z|}.\ee
The extension of the estimate (\ref{thk-1}) is probably of independent interest.
In particular, if the origin is a fixed point of $f$ in Theorem A, then Osserman
obtained a better estimate as follows.

 \begin{ThmB}\label{Thm-CPV-2}{\rm  (\cite[Lemma 2]{Os})}
Let $f$ be a holomorphic mapping  of $\mathbb{U}$ into itself with
$f(0)=0$. Then, for  $z\in\mathbb{U}$, the following sharp estimate
\be\label{Sch-O-2} |f(z)|\leq|z|\frac{|f'(0)|+|z|}{1+|f'(0)|\,|z|}
\ee holds.
\end{ThmB}

As an application,  by using (\ref{Sch-O-2}), Osserman \cite{Os}
established a version of the boundary Schwarz lemma which is as
follows.

\begin{ThmC}\label{Thm-Bound-1}
Let $f$ be a holomorphic mapping of $\mathbb{U}$ into itself with
$f(0)=0$. If $f$ is holomorphic at $b\in\partial\mathbb{U}$ $($or more generally, if $f$ is differentiable at $b\in\partial\mathbb{U}$$)$
and
$|f(b)|=1$, then $|f'(b)|\geq2/(1+|f'(0)|)$. Moreover, the
inequality is sharp.
\end{ThmC}

In fact, Unkelbach \cite{U-1938} had established a similar result as follows:
Let $f$ be a holomorphic mapping of $\mathbb{U}$ into itself with
$f(0)=0$. If $$D=\lim_{z\rightarrow1^{-}}\frac{1-f(z)}{1-z}$$ exists, then $D\geq1$. Moreover, if $f(0)=\varrho\,e^{i\varphi}$, then
$D\geq2(1-\rho\cos\varphi)/(1-\rho^{2})$. This inequality is also sharp.
On the related investigation of the boundary Schwarz type lemmas of the Poisson-Stieltjes integral, we refer to  \cite{H-1940}.

In the following, we extend Theorems A and B to
Banach spaces, and then we apply the obtained results to study the
Bloch type Theorem and the boundary Schwarz type lemmas.

\begin{Thm}\label{thm-h1}
Suppose that $B_{X}$ and $B_{Y}$ are the unit balls of the complex
Banach spaces $X$ and $Y$, respectively. Let $f:\, B_{X}\rightarrow
\overline{B_{Y}}$ be a holomorphic mapping. Then
\beqq\|f(z)\|_{Y}\leq\frac{\|f(0)\|_{Y}+\|z\|_{X}}{1+\|f(0)\|_{Y}\|z\|_{X}} \ \ \mbox{for}~z\in
B_{X}. \eeqq
This estimate is sharp
with equality possible for each value of $\| f(0)\|_Y$
and for each $z\in B_X$.
\end{Thm}

\begin{Thm}\label{thm-h2}
Suppose that $B_{X}$ and $B_{Y}$ are the unit balls of the complex
Banach spaces $X$ and $Y$, respectively. Let $f:\, B_{X}\rightarrow
B_{Y}$ be a holomorphic mapping with $f(0)=0$. Then
\beqq\|f(z)\|_{Y}\leq\frac{\|Df(0)\|+\|z\|_{X}}{1+\|Df(0)\|\|z\|_{X}}\|z\|_{X}\leq\|z\|_{X}
~\mbox{for}~z\in
B_{X}.
\eeqq
The first estimate is sharp
with equality possible for each value of $\| Df(0)\|$
and for each $z\in B_X$.
\end{Thm}

We use $\mathbf{B}$ to denote the homogeneous unit ball of
$X=\mathbb{C}^{n}$.
It is easy to see that
 $\mathbf{B}$ is the unit ball of a finite dimensional
JB$^*$-triple $X$. Let $k\in[0,1)$ be a constant.  Denote by
$\mathscr{PH}(k)$  the set of all  pluriharmonic mappings
$f=h+\overline{g}$ of $\mathbf{B}$ into $\mathbb{C}^{n}$ with
$h(0)=g(0)=0$ and
$$\left\|\omega_{f}\right\|\leq k,
$$
where  $h$ is locally biholomorphic in $\mathbf{B}$, $g$ is
holomorphic in $\mathbf{B}$,
$\omega_{f}=Dg[Dh]^{-1}$
and
\[
\|\omega_{f}\|=\sup_{z\in \mathbf{B}, \xi\in \C^n\backslash\{0\}}\frac{\|\omega_{f}(z)
\xi\|_{e}}{\|\xi\|_{e}}.
\]
For $n\geq2$, $f=h+\overline{g}\in\mathscr{PH}(k)$ is a quasiregular
mapping if and only if $h$ is a quasiregular
mapping (see \cite{CP-2017}).
In particular, if $n=1$, then $f\in\mathscr{PH}(k)$ is a
$K$-quasiregular mapping, where $K=(1+k)/(1-k)$ (cf.
\cite{CPW-2021,Du,Va,Vu}).

Denote by $\mathscr{P}$ a set of mappings from $\mathbf{B}$ into
$\mathbb{C}^{n}$.
  For a mapping $f\in\mathscr{P}$ and a point $a\in \mathbf{B}$, we write
$\mathscr{B}_{f}(a)$ as the radius of the largest univalent Euclidean ball
centered at $f(a)$ in $f(\mathbf{B})$. Here a univalent ball in
$f(\mathbf{B})$ centered at $f(a)$ means that $f$ maps an open
subset of $\mathbf{B}$ containing the point $a$ univalently onto
this ball. Let
$$\mathscr{B}=\inf_{f\in \mathscr{P}}\sup_{a\in \mathbf{B}}\mathscr{B}_{f}(a).
$$
If $\mathscr{B}>0$ is finite, then we call $\mathscr{B}$ the Bloch
type constant of the set $\mathscr{P}$. One of the long standing
open problems of determining the precise value of Bloch type
constant of holomorphic mappings with one variable  has attracted
much attention (see \cite{B,BE,GPS,Lan,LM}). For holomorphic
mappings of several complex variables, the Bloch type constant does
not exist unless one considers the class of
functions under certain constraints. For example, consider $f_{k}(z)=(kz_{1},z_{2}/k,z_{3},\ldots,z_{n})$ for $k\in \mathbb{N}=\{1,2,\ldots\}$,
where $n\geq2$ and $z$ is in the Euclidean unit ball $\mathbb{B}^{n}$ of $\mathbb{C}^{n}$.
It is easy to see that each $f_{k}$ is univalent and $|f_{k}(0)|=\det Df_{k}(0)-1=0$. Moreover, each
$f_{k}(\mathbb{B}^{n})$ contains no ball with radius bigger than $1/k$. Hence, there does not exist an absolute
constant $r_{0}$ which can work for all $k\in \mathbb{N}$ such that
 $\{z\in\mathbb{C}^{n}:\, \|z\|_{e}<r_{0}\}$ is contained in $f_{k}(\mathbb{B}^{n})$.
For more details on  studies of the Bloch type constant of holomorphic mappings with several complex variables, we refer   to  the works of
Chen and Gauthier \cite{CG01},
Fitzgerald and Gong \cite{FG},
Graham and Varolin \cite{GV}, Hamada \cite{H19JAM},
Takahashi \cite{Ta},
and
Wu \cite{Wu}.
 On  the
studies of the Bloch type constant for the class of pluriharmonic
mappings, we refer to \cite{CG11,HK15}.

In the following, for
$f=h+\overline{g}\in \mathscr{PH}(k)$, we will use Theorems
\ref{thm-h1} and \ref{thm-h2} to investigate the ratio
$\mathscr{B}_{f}/\mathscr{B}_{h}$ and give a sharp estimate. For the
related studies of the planar harmonic mappings, see
\cite{CG11,CLW,GPS}.

\begin{Thm}\label{thm-Bloch}
 For $k\in[0,1)$, let $f=h+\overline{g}\in
\mathscr{PH}(k)$.
Then, for $z\in
\mathbf{B}$,

\be\label{eq-thm-3.0}
1-k\leq\frac{\mathscr{B}_{f}(z)}{\mathscr{B}_{h}(z)}\leq\mu_{k}\left(\frac{\|\omega_{f}(z)\|}{k}\right)\leq\mu_{k}(1)=1+k,
\ee where
$$ \begin{cases}
\mu_{k}(x)= \displaystyle 1+
k\left[\frac{1}{x}+\left(1-\frac{1}{x^{2}}\right)\log(1+x)\right] &~\mbox{ for $x \in (0,1]$},\\[3mm]
\mu_{k}(x)=\displaystyle
\lim_{x\rightarrow0^{+}}\mu_{k}(x)=1+\frac{k}{2} &~\mbox{ for $x=0$}.
\end{cases}
$$

Moreover, the left hand of  {\rm (\ref{eq-thm-3.0})} is sharp for
all $z\in \mathbf{B}$, and the right hand of {\rm
(\ref{eq-thm-3.0})} is asymptotically sharp when $k$ tends to $0$.
\end{Thm}

In \cite{Wu},  Wu  generalized the classical Schwarz lemma of
holomorphic mappings to higher dimension. Burns and Krantz \cite{BK}
established a new version Schwarz lemma at the boundary, and
obtained a new rigidity result for holomorphic mappings. Later,
Huang \cite{Hu} further strengthened the result of  Burns-Krantz for
holomorphic mappings with an interior fixed point. Recently, Liu and Tang \cite{L-T} obtained a
Schwarz lemma at the boundary of holomorphic mappings from a Levi strongly pseudoconvex domain into itself.
See
\cite{GHK-JAM, H18,Hu-1,T} for more details on this line. In the following, we
will apply Theorem \ref{thm-h2} to establish a  new version Schwarz
lemma at the boundary, which is a generalization of Theorem
C.

\begin{Thm}\label{thm-3h}
Let $B_{X}$ and $B_{Y}$ be the unit balls of the complex
Banach spaces $X$ and $Y$, respectively.
Suppose that $f$ is a holomorphic mapping of $B_X$ into $B_Y$.
  If
$f(0)=0$ and $f$ is holomorphic at $b\in\partial B_X$
$($or more generally, the radial derivative $Df(b)b$ exists at $b\in\partial B_X$$)$
with
$\|f(b)\|_{Y}=1$, then \beqq \|Df(b)b\|_Y\geq\frac{2}{1+\|Df(0)\|}.
\eeqq
This inequality is sharp
with equality possible for each value of $\| Df(0)\|$.
\end{Thm}

In particular, if we replace $B_{X}$ and $B_{Y}$ by a balanced domain and
a finite dimensional bounded symmetric domain in Theorem \ref{thm-3h},
respectively, then we obtain a better estimate
 (cf. \cite{HKo}).
 Before we present the next result, let us recall some definitions.

Let $\B_Y$ be a bounded symmetric domain
realized as the open unit ball of a finite dimensional JB$^*$-triple
$Y$.
We recall a constant $c(\B_Y)$ defined in \cite{HHK}.
Let $h_0$ be the Bergman metric on $\B_Y$ at $0$
and let
\[
c(\B_Y)=\frac{1}{2}\sup_{x,y\in \B_Y}|h_0(x,y)|.
\]

\noindent It follows from \cite[Ineq. (2.3)]{CHHK16} that
$$\frac{\dim Y+r}{2}\leq c(\B_Y)\leq\dim Y,$$ where $r$ is the rank of $Y$.

An element $x$ in a JB$^*$-triple $Y$ is called a tripotent if $x$ satisfies $\{x,x,x\}=x$. If
two tripotents $x$ and $y$ satisfy $2x\square y=0$, then  $x$ and $y$   are said to be orthogonal. Obviously,
orthogonality is a symmetric relation. A tripotent $x$ is said to be maximal if any tripotent  which is
orthogonal to $x$ is $0$.

\begin{Thm}\label{thm-Sy}
Suppose that
 $G$ is a balanced domain in a complex Banach space $X$ and $\B_Y$ is a bounded symmetric domain
realized as the open unit ball of a finite dimensional JB$^*$-triple
$Y$.
Let $\Gamma \subset \partial\B_Y$ be the set of maximal tripotents of $Y$, and
let
$f:\, G\to \B_Y$ be a holomorphic mapping.
Also let $f$ be holomorphic at $z=\alpha\in \partial G$ and $f(\alpha)=\beta\in \Gamma$.
\begin{enumerate}
	\item[{\rm (i)}] We have
\begin{equation}
\label{eq-Schwarz-symmetric}
\frac{1}{2c(\B_Y)}
h_0(Df(\alpha)\alpha, \beta)\geq
\frac{2\left| 1-\frac{1}{2c(\B_Y)}h_0(f(0),\beta)\right|^2}
{1-\left|\frac{1}{2c(\B_Y)}h_0(f(0),\beta)\right|^2+\| Df(0)\alpha\|_Y},
\end{equation}
where $h_0$ is the Bergman metric on $\B_Y$ at $0$.

	\item[{\rm (ii)}] Moreover, if $f(0)=0$, then we have
\begin{equation}
\label{eq-Schwarz-symmetric2}
\frac{1}{2c(\B_Y)}
h_0(Df(\alpha)\alpha, \beta)\geq
\frac{2}
{1+\| Df(0)\alpha\|_Y}.
\end{equation}

	\item[{\rm (iii)}] In particular, if $G=B_{X}$ is the unit ball of $X$,
then the inequalities
$(\ref{eq-Schwarz-symmetric})$ and
$(\ref{eq-Schwarz-symmetric2})$
are sharp
with equality possible for each values of
$$a=\frac{1}{2c(\B_Y)} h_0(f(0),\beta),~b=\frac{1}{2c(\B_Y)} h_0(Df(0)\alpha,\beta)$$
with $|b|\leq 1-|a|^2$.
\end{enumerate}
\end{Thm}

We remark that Theorem \ref{thm-Sy} is an improvement and generalization of \cite[Theorem 3.1]{L-T-2015} and \cite[Theorem 3.1]{L-T-2020}
(cf. \cite[Theorem 1.5]{Z18}).
By using arguments similar to those in the proof of Theorem \ref{thm-Sy}, we have the following theorem
(cf. \cite[Theorem 1.5]{Z18}).
We omit the proof.

\begin{Thm}\label{thm-Hilbert}
Suppose that
 $G$ is a balanced domain in a complex Banach space $X$ and $\B_H$ is unit ball of a complex Hilbert space
$H$ with inner product $\langle \cdot, \cdot \rangle$.
Let
$f:\, G\to \B_H$ be a holomorphic mapping.
If $f$ is holomorphic at $z=\alpha\in \partial G$ and $f(\alpha)=\beta\in \partial \B_H$,
then
\begin{equation}
\label{eq-Schwarz-Hilbert}
\langle Df(\alpha)\alpha, \beta\rangle\geq
\frac{2\left| 1-\langle f(0),\beta\rangle \right|^2}
{1-\left|\langle f(0),\beta \rangle \right|^2+\| Df(0)\alpha\|_H}.
\end{equation}

Moreover, if $f(0)=0$, then we have
\begin{equation}
\label{eq-Schwarz-Hilbert2}
\langle Df(\alpha)\alpha, \beta\rangle
\geq
\frac{2}
{1+\| Df(0)\alpha\|_H}.
\end{equation}

In particular, if $G=B_{X}$ is the unit ball of $X$,
then the inequalities
$(\ref{eq-Schwarz-Hilbert})$ and
$(\ref{eq-Schwarz-Hilbert2})$
are sharp
with equality possible for each values of $\langle f(0),\beta\rangle$ and $\| Df(0)\alpha\|_H$
with $\| Df(0)\alpha\|_H \leq 1-|\langle f(0),\beta\rangle|^2$

\end{Thm}

The proofs of Theorems \ref{thm-h1}$\sim$\ref{thm-Sy}
will be presented in part I of Section \ref{sec-4}.

\section{Schwarz  type lemmas of pluriharmonic mappings and their applications}

Heinz in his classical paper \cite{Hei} showed that the following version of Schwarz  Lemma for harmonic mappings: If $f$ is a harmonic mapping of
 $\U$ into itself with
$f(0)=0$, then
\be\label{eq-Heinz}|f(z)|\leq\frac{4}{\pi}\arctan|z|\ee
for $z\in\U$. In
2011, Chen and Gauthier generalized (\ref{eq-Heinz}) into the
following form.

\begin{ThmD}{\rm (\cite[Theorem 4]{CG11})}\label{Thm-A-1}
Let $f$ be a pluriharmonic mapping of the Euclidean unit ball $\mathbb{B}^{n}$
into the Euclidean unit ball $\mathbb{B}^{m}$ such that $f(0)=0$, where $m$ is a
positive integer. Then, for all $z\in\mathbb{B}^{n}$,
$$\|f(z)\|_{e}\leq\frac{4}{\pi}\arctan\|z\|_{e}.$$
This estimate is sharp.
\end{ThmD}
Hamada and Kohr \cite{HK15} extended Theorem D to
pluriharmonic mappings of the unit ball of a complex Banach space
$X$ into the unit ball $\mathbb{B}^{n}_{a}$ of $\C^n$ with respect to
an arbitrary norm $\| \cdot \|_a$
on $\C^n$
as follows.

\begin{ThmE}{\rm (\cite[Theorem 4.1]{HK15})}\label{Thm-A-2}
Let $B_{X}$ be the unit ball of a complex Banach space $X$,
{$\mathbb{B}^{n}_{a}$} be the unit ball of $\C^n$ with respect to
an arbitrary norm $\| \cdot \|_a$ on $\C^n$ and
{$f:\, B_{X}\rightarrow\mathbb{B}^{n}_{a}$} be a pluriharmonic
mapping such that $f(0)=0$. Then, for  $z\in B_{X}$, the following sharp inequality
$$\|f(z)\|_{a}\leq\frac{4}{\pi}\arctan\|z\|_{X}$$ holds.
\end{ThmE}

We remove the assumption ``$f(0)=0$" in Theorem E and
obtain the following result.

\begin{Thm}\label{Har-1}
Suppose that $B_{X}$ and $B_{Y}$ are the unit balls of the complex
Banach spaces $X$ and $Y$, respectively, and $f:\, B_{X}\rightarrow
B_{Y}$ is a pluriharmonic mapping. Then
$$
\left\|f(z)-\frac{1-\|z\|_{X}^{2}}{1+\|z\|_{X}^{2}}f(0)\right\|_{Y}\leq\frac{4}{\pi}\arctan\|z\|_{X} ~\mbox{ for $z\in B_{X}$}.
$$
In particular, if $f(0)=0$, then this estimate is sharp.
\end{Thm}

In particular, if $f(0)=0$ in Theorem \ref{Har-1}, then we have a better estimate as follows
(cf. \cite[Theorem 1.7]{Z18}).

\begin{Thm}\label{Har-2}

Assume the hypothesis of Theorem \ref{Har-1}, and in addition let $f(0)=0.$ Then we have
$$\|f(z)\|_{Y}\leq\frac{4}{\pi}\arctan\left(\frac{\|z\|_{X}+
\frac{\pi}{4}\Lambda_{f}(0; w)}{1+\frac{\pi}{4}\Lambda_{f}(0; w)\|z\|_{X}}\|z\|_{X}\right)\leq\frac{4}{\pi}\arctan\|z\|_{X}
~\mbox{ for $z\in B_{X}$},
$$
where $w=z/\| z\|_X$.
\end{Thm}

Let $f$ be a one-to-one harmonic mapping  of $\U$ onto itself
with $f(0)=0$
which is $C^1$ up to the boundary.
By using (\ref{eq-Heinz}),  Heinz  \cite[Ineq. (15)]{Hei} proved that, for any
$\theta\in[0,2\pi]$,
\be\label{eq-hgj-1}|f_{\zeta}(e^{i\theta})|+|f_{\overline{\zeta}}(e^{i\theta})|\geq\frac{2}{\pi}.\ee
In the following, we extend (\ref{eq-hgj-1}) into the following forms.

First, we
will apply Theorem \ref{Har-1} and the Harnack principle to establish a  new version Schwarz
lemma at the boundary for pluriharmonic mappings.

\begin{Thm}\label{thm-Har3a}

Assume the hypotheses of Theorem \ref{Har-1}. In addition, assume that the radial derivative $Df(b)b$
exists at $b\in\partial B_X$  with $\|f(b)\|_{Y}=1.$ Then we have
\beqq \|Df(b)b\|_Y\geq
\max\left\{ \frac{2}{\pi}-\| f(0)\|_Y, \frac{1-\| f(0)\|_Y}{2} \right\}.
\eeqq
\end{Thm}

Next, we
will apply Theorem \ref{Har-2} to establish a  new version Schwarz
lemma at the boundary for pluriharmonic mappings.
For the proof, it suffices to use  arguments similar to those in the proof of Theorem \ref{thm-3h}.
We omit the proof.

\begin{Thm}\label{thm-Har3}

Assume the hypothesis of Theorem \ref{thm-Har3a}.
Then, if $f(0)=0,$ we have
\beqq
\|Df(b)b\|_Y\geq
\frac{4}{\pi}\frac{1}{1+\frac{\pi}{4}\Lambda_f(0;b)}.
\eeqq
\end{Thm}

In particular, we consider the cases such that $B_Y$ is a bounded symmetric domain in $\C^n$ or is the complex Hilbert ball
(cf. \cite[Theorems 1.8 and 1.12]{Z18}).
In these cases, the domain of the definition of the mapping $f$ can be generalized to a balanced domain in a complex Banach space.

\begin{Thm}\label{Schwarz-pluriharmonic-symmetric}

Suppose that
 $G$ is a balanced domain in a complex Banach space $X$ and $\B_Y$ is a bounded symmetric domain
realized as the open unit ball of a finite dimensional JB$^*$-triple
$Y=\C^n$.
Let $\Gamma \subset \partial\B_Y$ be the set of maximal tripotents of $Y$, and
let
$f:\,G\to \B_Y$ be a pluriharmonic mapping with $f(0)=0$.
Assume that $f$ is differentiable at $z=\alpha\in \partial G$ and $f(\alpha)=\beta\in \Gamma$.
Let
\[
\varphi(\zeta)=\frac{1}{2c(\B_Y)}h_0(f(\zeta \alpha),\beta),
\quad \zeta \in \U,
\]
where $h_0$ is the Bergman metric on $\B_Y$ at $0$.
Then $\varphi$ is harmonic mapping of
$\U$ into itself with
$\varphi(0)=0$ and
we have
\begin{equation}
\label{eq-Schwarz-pluriharmonic-symmetric}
{\rm Re} \big(
\varphi_{\zeta}(1)+\varphi_{\overline{\zeta}}(1)
\big)\geq
\frac{4}{\pi}
\frac{1}
{1+\frac{\pi}{4}\Lambda_{f}(0,\alpha)},
\end{equation}
where ${\rm ``Re"}$ denotes the real part of a complex number.

In particular,
if $f=h+\overline{g}$,
where $h$ and $g$ are holomorphic on $G$,
and $f$ satisfies the above assumptions,
then we have
\begin{equation}
\label{eq-Schwarz-pluriharmonic-symmetric2}
\frac{1}{2c(\B_Y)}
{\rm Re} \big(h_0(Dh(\alpha)\alpha+
\overline{Dg(\alpha)\alpha}, \beta)\big)\geq
\frac{4}{\pi}
\frac{1}
{1+\frac{\pi}{4}\Lambda_{f}(0,\alpha)}.
\end{equation}
If $G=B_{X}$ is the unit ball of $X$,
then the inequalities
$(\ref{eq-Schwarz-pluriharmonic-symmetric})$ and
$(\ref{eq-Schwarz-pluriharmonic-symmetric2})$
are sharp.
\end{Thm}

If $B_{Y}$ is the complex Hilbert ball, then we have the following result.
We omit the proof, since it is similar to that in the proof of Theorem \ref{Schwarz-pluriharmonic-symmetric}.

\begin{Thm}\label{Schwarz-pluriharmonic-Hilbert}

Suppose that
 $G$ is a balanced domain in a complex Banach space $X$ and $\B_H$ is the unit ball of a complex Hilbert space $H$
 with inner product $\langle \cdot, \cdot \rangle$.
Let
$f:\,G\to \B_H$ be a pluriharmonic mapping with $f(0)=0$.
Assume that $f$ is differentiable at $z=\alpha\in \partial G$ and $f(\alpha)=\beta\in \partial \B_H$.
Let
\[
\varphi(\zeta)=\langle f(\zeta \alpha),\beta\rangle,
\quad \zeta \in \U.
\]
Then $\varphi$ is a harmonic mapping of
 $\U$ into itself with
$f(0)=0$ and
\begin{equation}
\label{eq-Schwarz-pluriharmonic-Hilbert}
{\rm Re} \big(
\varphi_{\zeta}(1)+\varphi_{\overline{\zeta}}(1)
\big)\geq
\frac{4}{\pi}
\frac{1}
{1+\frac{\pi}{4}\Lambda_{f}(0,\alpha)}.
\end{equation}

In particular,
if $H=\ell_2$ and $f=h+\overline{g}$,
where $h$ and $g$ are holomorphic on $G$,
and $f$ satisfies the above assumptions,
then we have
\begin{equation}
\label{eq-Schwarz-pluriharmonic-Hilbert2}
{\rm Re} \langle Dh(\alpha)\alpha+
\overline{Dg(\alpha)\alpha}, \beta \rangle\geq
\frac{4}{\pi}
\frac{1}
{1+\frac{\pi}{4}\Lambda_{f}(0,\alpha)}.
\end{equation}
The inequalities
$(\ref{eq-Schwarz-pluriharmonic-Hilbert})$ and
$(\ref{eq-Schwarz-pluriharmonic-Hilbert2})$
are sharp.
\end{Thm}

The
classical Schwarz-Pick lemma states that an analytic function $f$
of $\mathbb{U}$ into itself satisfies

\begin{equation}\label{eq-sp}
|f'(z)|\leq\frac{1-|f(z)|^{2}}{1-|z|^{2}},~z\in\mathbb{U}.
\end{equation}
Chu et al. \cite[Lemma 3.12]{CHHK17}
generalized (\ref{eq-sp}) to holomorphic mappings between the unit balls of
JB$^*$-triples
(cf. \cite{BGM14}, \cite{CG01}, \cite{HK15}).

In 1989, Colonna established an analogue of the Schwarz-Pick lemma
for planar harmonic mappings.

\begin{ThmF}{\rm (\cite[Theorems 3 \mbox{and} 4]{C89})}\label{Coo}
Let $f$ be a harmonic mapping of $\mathbb{U}$ into itself.
Then, for $z\in\mathbb{U}$,
$$|f_{z}(z)|+|f_{\overline{z}}(z)|\leq\frac{4}{\pi}\frac{1}{1-|z|^{2}}.
$$
This estimate is sharp, and all the extremal functions are
$$f(z)=\frac{2\gamma }{\pi}\arg \left (
\frac{1+\phi(z)}{1-\phi(z)}\right), $$ where $|\gamma|=1$ and $\phi$
is a conformal automorphism of $\mathbb{U}$.
\end{ThmF}

For real valued harmonic functions in $\U$,
Kalaj and Vuorinen \cite{KV12}
and
Chen \cite[Theorem 1.2]{C13}
obtained a better estimate as follows.

\begin{ThmG}
\label{thm-KV12}
Let $f:\,\U\to (-1,1)$ be a real-valued harmonic function.
Then the following inequalities hold:
\begin{enumerate}
\item $\displaystyle
|\nabla f(z)|\leq \frac{4}{\pi}\frac{1-|f(z)|^2}{1-|z|^2},$ for $z\in \U.$
This inequality is sharp for each $z\in\U$.

\item $\displaystyle
|\nabla f(z)|\leq \frac{4}{\pi}\frac{\cos(\frac{\pi}{2}f(z))}{1-|z|^2},$ for $z\in \U.$
This inequality is sharp for each $z\in\U$.

\end{enumerate}
\end{ThmG}

Chen and Gauthier \cite{CG11} generalized Theorem
F to pluriharmonic mappings between the
Euclidean unit balls. Later, Hamada and Kohr \cite{HK15} generalized Theorem
F to pluriharmonic mappings from  the unit ball
of a JB$^*$-triple into the unit ball {$\mathbb{B}^n_{a}$} of $\C^n$ with respect
to an arbitrary norm on $\C^n$.

For mappings with values in higher dimensional spaces,
Pavlovi\'c \cite{Pav1,P11} showed that the inequality (\ref{eq-sp}) does not hold for analytic functions $f$ of
$\mathbb{U}$ into $\mathbb{B}^{n}$, where $n\geq2$ is an integer.
However, Pavlovi\'c  proved the following Schwarz-Pick type lemma for analytic functions $f$ of
$\mathbb{U}$ into $\mathbb{B}^{n}$:

\[
|\nabla\|f(z)\|_{e}|\leq\frac{1-\|f(z)\|_{e}^{2}}{1-|z|^{2}},
\quad z\in\mathbb{U},
\]
where $\nabla\|f(z)\|_{e}$ denotes the  gradient of $\|f\|_{e}$.

In \cite{Z19}, Zhu established a  Schwarz-Pick type estimate for pluriharmonic mappings $f$ of
$\mathbb{B}^{n}$ into itself as follows.

\begin{ThmH}{\rm (\cite[Theorem 1.1]{Z19})}\label{Z-1}
For $n\geq1$, let $f$ be a pluriharmonic mapping of $\mathbb{B}^{n}$ into itself. Then the following inequality
$$|\nabla\|f(z)\|_{e}|\leq\frac{4\sqrt{n}}{\pi}\frac{1}{ (1-\|z\|_{e})}$$ holds for all $z\in\mathbb{B}^{n}.$
\end{ThmH}

In \cite{CH20}, Chen and Hamada  improved Theorem H into the  sharp form for pluriharmonic mappings
of $\mathbb{B}^{n}$ into the unit ball of the Minkowski space.
In the following,
we will present a   relatively simple  method of proof to improve Theorem H and \cite[Theorem 2.2]{CH20}
to pluriharmonic mappings from the unit ball $\B_X$ of a
JB$^*$-triple $X$ to the unit ball of  a complex Banach space $Y$.

\begin{Thm}\label{Schwarz-norm}
Let $\B_X$ be a bounded symmetric domain realized as the unit ball of a JB$^*$-triple $X$.
Let $B_Y$ be the unit ball of a complex Banach space $Y$.
Also, let $f:\,\B_X\to B_Y$ be a pluriharmonic mapping.
Let $z_0\in \B_X$ be a point which satisfies one of the following conditions:
\begin{enumerate}
\item[(i)]
$f(z_0)=0$;
\item[(ii)]
$f(z_0)\neq 0$ and $\|f(z)\|_Y$ is differentiable at $z=z_0$.
\end{enumerate}
Then
\begin{equation}
\label{P-Schwarz-Pick-norm}
|\nabla \| f\|_Y(z_0)|\leq
\frac{4}{\pi}\frac{1}{1-\| z_0\|_X^2}.
\end{equation}
The estimate
{\rm (\ref{P-Schwarz-Pick-norm})} is sharp
for each $z_0\in \B_X$.
\end{Thm}

In particular, if $f$ is a pluriharmonic mapping of $\B_X$ into the unit ball $\mathbf{B}_Y$ of a real Banach space $Y$, then we have
a better estimate as follows.
Here we omit the proof because it suffices to use Theorem F  instead of Theorem G and use  arguments similar to those in the proof of Theorem \ref{Schwarz-norm}.

\begin{Thm}\label{R-CH}
\label{Schwarz-norm-real}
Let $\B_X$ be a bounded symmetric domain realized as the unit ball of a JB$^*$-triple $X$.
Let $\mathbf{B}_Y$ be the unit ball of a real Banach space $Y$.
Also, let $f:\,\B_X\to \mathbf{B}_Y$ be a pluriharmonic mapping.
Let $z_0\in \B_X$ be a point which satisfies one of the following conditions:
\begin{enumerate}
\item[(i)]
$f(z_0)=0$;
\item[(ii)]
$f(z_0)\neq 0$ and $\|f(z)\|_Y$ is differentiable at $z=z_0$.
\end{enumerate}
Then  we have the following estimates.
\begin{enumerate}
\item[(1)]
$\displaystyle  |\nabla \| f\|_Y(z_0)|\leq \frac{4}{\pi}\frac{1-\| f(z_0)\|_Y^2}{1-\| z_0\|_X^2}.$

\item[(2)]
$\displaystyle |\nabla \| f\|_Y(z_0)|\leq
\frac{4}{\pi}\frac{\cos(\frac{\pi}{2}\| f(z_0)\|_Y)}{1-\| z_0\|_X^2}.
$
\end{enumerate}
The above estimates are sharp for each $z_0\in \B_X$.
\end{Thm}

We remark that Theorem \ref{R-CH} is an improvement and generalization of \cite[Theorem 2.3]{CH20}.

In \cite{Z19},
Zhu established some other Schwarz-Pick type estimates as follows.

\begin{ThmI} {\rm (\cite[Theorem 1.2]{Z19})}
\label{Zhu1.2}
Let $f=h+\overline{g}$ be a pluriharmonic self-mapping of the unit ball $\mathbb{B}^n$ of $\C^n$,
where $h$ and $g$ are holomorphic mappings of  $\mathbb{B}^n$  into $\C^n$.
Then
\be
\label{1.4j}
\| Dh(z)e_j\|_e^2+\| Dg(z)e_j\|_e^2
\leq
\frac{1-\| f(z)\|_e^2}{(1-\| z\|_e)^2},
\quad
j=1,\ldots, n,
\ee
where $e_1, \dots, e_n$ is the usual orthonormal basis of $\C^n$,
and
\be
\label{1.5j}
\|\nabla f(z)\|_F^2
\leq
\frac{2n(1-\| f(z)\|_e^2)}{(1-\| z\|_e)^2}.
\ee
\end{ThmI}

\begin{ThmJ}{\rm (\cite[Theorem 1.3]{Z19})}
\label{Zhu1.3}
Let $f=h+\overline{g}$ be a pluriharmonic self-mapping of $\mathbb{B}^{n}$,
where $h$ and $g$ are holomorphic mappings of $\mathbb{B}^{n}$
into $\C^n$.
Assume that $h$ is locally biholomorphic
on $\mathbb{B}^{n}$ and
$\| \omega_f\|\leq k<1$ holds on $\mathbb{B}^{n}$.
Let $K=(1+k)/(1-k)$.
Then the following inequalities hold for all $z\in \mathbb{B}^{n}$:
\be
\label{1.6}
\| Dh(z)\|+\|Dg(z)\|
\leq
\frac{2K}{K+1}\frac{\sqrt{n(1-\| f(z)\|_e^2)}}{1-\| z\|_e},
\ee
\be
\label{1.7}
|\nabla \| f\|_e(z)|
\leq
\sqrt{n}(\|Dh(z)\|+\|Dg(z)\|)
\ee
and
\be
\label{1.8}
\|\nabla  f(z)\|_F
\leq
\sqrt{2n(\|Dh(z)\|^2+\|Dg(z)\|^2)}.
\ee
\end{ThmJ}

Chen and Hamada \cite{CH20}  generalized and improved Theorem I into the following sharp form on pluriharmonic
mappings of the polydisc $\U^{n}$ in $\mathbb{C}^{n}$ into the Euclidean unit ball $\mathbb{B}^{m}$ in $\mathbb{C}^{m}$.

\begin{ThmK}\label{thm-3+}
For $n\geq1$, let
$f=(f_{1},\ldots,f_{m}):\, \U^{n}\to\mathbb{B}^{m}$ be a pluriharmonic
mapping, where   $m$ is a positive integer.
Then, for $z=(z_{1},\ldots,z_{n})\in\U^{n}$, we have
\begin{equation}\label{qq-ch1.01g}
\sum_{j=1}^{m}\sum_{k=1}^{n}\left(\left|\frac{\partial f_{j}(z)}{\partial z_{k}}\right|^{2}+
\left|\frac{\partial f_{j}(z)}{\partial\overline{z}_{k}}\right|^{2}\right)
\leq\frac{1-\|f(z)\|_{e}^{2}}{(1-\|z\|_{\infty}^{2})^{2}},
\end{equation} where $\|z\|_{\infty}=\max_{1\leq j\leq n}|z_{j}|$.
Moreover, the inequality {\rm (\ref{qq-ch1.01g})} is sharp for each $z\in \U^{n}$
with $|z_1|=\cdots =|z_n|$.
\end{ThmK}

In the following, we will generalize and   improve
\eqref{1.7} into the following form.
Moreover, we only need to assume that $f$ is pluriharmonic.

\begin{Pro}
\label{prop}
Let $\Omega$ be a domain in a complex Banach space $X$,
 and $f=h+\overline{g}:\, \Omega \to \C^n$ be a pluriharmonic mapping,
where $h$ and $g$ are holomorphic mappings
of $\Omega$ into the Euclidean space $\C^n$.
Then
$$
|\nabla \| f\|_e(z)|\leq \| Dh(z)\|+\| Dg(z)\|,
\quad z\in \Omega.
$$
\end{Pro}

\begin{proof}
For a pluriharmonic mapping
$f=h+\overline{g}:\, \Omega \to \C^n$,
we have (see e.g. \cite[p.638]{HK15}),
\[
\| Df(z)\| \leq
\| Dh(z)\|+\| Dg(z)\|,
\quad
z\in \Omega.
\]
Then, by (\ref{nabla-norm-differentiable}), we have
\[
|\nabla \| f\|_e(z)|\leq \| Df(z)\| \leq\| Dh(z)\|+\| Dg(z)\|, \ z\in \Omega.
\]
This completes the proof.
\end{proof}

Let $X$ be the $\ell^\infty$-sum $X= X_1 \oplus \cdots \oplus X_m$ of
JB*-triples $X_1$, \ldots, $X_m$
and
let $\B_X=\B_{X_1}\times \cdots\times \B_{X_m}$ be the unit ball of $X$.
By applying Theorem K, we obtain
the following improvement of
\eqref{1.4j} to
pluriharmonic mappings of the unit ball $\B_X=\B_{X_1}\times \cdots\times \B_{X_m}$
into the Euclidean unit ball $\mathbb{B}^{n}$ of $\C^n$.

\begin{Thm}
\label{Schwarz-Hilbert}
Let $X$ be the $\ell^\infty$-sum $X= X_1 \oplus \cdots \oplus X_m$ of
JB*-triples $X_1$, \ldots, $X_m$
and $\B_X=\B_{X_1}\times \cdots\times \B_{X_m}$ be the unit ball of $X$.
Also, let $f=h+\overline{g}:\, \B_X\to \mathbb{B}^{n}$ be a pluriharmonic mapping,
where
$h$ and $g$ are holomorphic mappings of $\B_X$ into $\C^n$.
Then we have
\begin{equation}
\label{P-Schwarz-Pick-Hilbert}
\sum_{j=1}^m \left(\| Dh(z)\widetilde{w}_j\|_e^2+\| Dg(z)\widetilde{w}_j\|_e^2\right)
\leq
\frac{1-\| f(z)\|^2_e}{(1-\| z\|_X^2)^2}
\end{equation}
for all
$w_j\in X_j$ with $\| w_j\|_{X_j}=1$ $(1\leq j\leq m)$,
where $\widetilde{w}_j=I_j(w_j)$ and $I_j:\,X_j\to X$ is the natural inclusion mapping
for $j=1,\dots, m$.

Moreover, the inequality {\rm (\ref{P-Schwarz-Pick-Hilbert})} is sharp for each $z=(z_1,\dots, z_m)\in \B_X$
with $\|z_1\|_{X_1}=\cdots =\|z_m\|_{X_m}$.
\end{Thm}

In particular, for $\B_X=\B^{k_1}\times \cdots \times \B^{k_p}$,
we have the following generalization of (\ref{qq-ch1.01g}) to pluriharmonic mappings of $\B_X$ into $\mathbb{B}^{n}$,
where $ \B^{k_j}$ are the Euclidean unit balls in $\mathbb{C}^{k_j}$ for $j\in\{1,\ldots,p\}$.

\begin{Cor}
\label{product-balls}
Let $\B_X=\B^{k_1}\times \cdots \times \B^{k_p}$.
Also, let $f=h+\overline{g}:\, \B_X\to \mathbb{B}^{n}$ be a pluriharmonic mapping,
where $h$ and $g$ are holomorphic mappings of $\B_X$ into $\C^n$.
Then
\begin{equation}
\label{product-ball-Schwarz-Pick}
\sum_{j=1}^{n}\sum_{k=1}^{m}\left(\left|\frac{\partial
f_{j}(z)}{\partial z_{k}}\right|^{2}+ \left|\frac{\partial
f_{j}(z)}{\partial\overline{z}_{k}}\right|^{2}\right) \leq
\frac{\kappa(1-\| f(z)\|^2_e)}{(1-\| z\|_X^2)^2},
\end{equation}
where $m=k_1+\cdots+k_p$ and $\kappa=\max\{ k_1, \cdots, k_p\}$.
\end{Cor}

\begin{proof}
By using the natural inclusion map $\zeta_j\in
\B^{k_j}\to (\zeta_j,0)\in \mathbb{B}^{\kappa}$,
we may assume that $\kappa=k_1=\cdots=k_p$. Let $e_1, \dots,
e_\kappa$ be the usual orthonormal basis of $\C^\kappa$ and let
$\widetilde{e}_{j,l}=I_j(e_l)$, where $I_j:\,X_j \to X$ is the natural
inclusion mapping. Then we have $\| e_{\nu}\|_{X_j}=1$ for each $\nu$
with $1\leq \nu\leq \kappa$ and each $j$ with $1\leq j\leq p$. By
\eqref{P-Schwarz-Pick-Hilbert}, we have
\[
\sum_{j=1}^p \left(\| Dh(z)\widetilde{e}_{j,l}\|_e^2+\|
Dg(z)\widetilde{e}_{j,l}\|_e^2\right) \leq \frac{1-\| f(z)\|^2_e}{(1-\|
z\|_X^2)^2}, \quad 1\leq l\leq \kappa,
\]
which gives that
\begin{eqnarray*}
\sum_{j=1}^{n}\sum_{k=1}^{m}\left(\left|\frac{\partial
f_{j}(z)}{\partial z_{k}}\right|^{2}+ \left|\frac{\partial
f_{j}(z)}{\partial\overline{z}_{k}}\right|^{2}\right) &=&
\sum_{l=1}^\kappa \sum_{j=1}^p \left(\|
Dh(z)\widetilde{e}_{j,l}\|_e^2+\| Dg(z)\widetilde{e}_{j,l}\|_e^2\right)
\\
&\leq & \frac{\kappa(1-\| f(z)\|^2_e)}{(1-\| z\|_X^2)^2}.
\end{eqnarray*}
This completes the proof.
\end{proof}

The following proposition is a generalization of
\eqref{1.8} to
pluriharmonic mappings of a domain $\Omega$ in the Euclidean space $\C^m$ into the Euclidean space $\C^n$.
Note that we only need to assume that $f$ is pluriharmonic.

\begin{Pro}
\label{prop1.8}
Let $\Omega$ be a domain in the Euclidean space $\C^m$.
Let $f=h+\overline{g}:\, \Omega\to \C^n$ be a pluriharmonic mapping,
where $h$ and $g$ are holomorphic mappings from $\Omega$
to the Euclidean space $\C^n$.
Then,  we have
\begin{equation}
\label{nabla-operator_norm}
\| \nabla f(z)\|_F
\leq
\sqrt{2m(\|Dh(z)\|^2+\|Dg(z)\|^2)}.
\end{equation}
\end{Pro}

\begin{proof}
 The inequality \eqref{nabla-operator_norm} follows from the relation
\begin{equation}
\label{Frobenius2}
\| \nabla f(z)\|_F^2=2(\| Dh(z)\|_F^2+\| Dg(z)\|_F^2)
\end{equation}
and  the inequality \eqref{A-Frobenius-Operator}. This completes the proof.
\end{proof}

The following corollary is a generalization of
\eqref{1.5j} to
pluriharmonic mappings from $\B_X=\B^{k_1}\times \cdots \times \B^{k_p}$ to $\mathbb{B}^{n}$.
Note that \eqref{Frobenius-Schwarz-Pick-Hilbert}
an improvement of \eqref{1.5j}.
In particular, the following result holds for
$\B_X=\B^{m}$ or for $\B_X=\U^m$.

\begin{Cor}
\label{Frobenius norm}

Assume the hypotheses of Corollary \ref{product-balls}. Then
\begin{equation}
\label{Frobenius-Schwarz-Pick-Hilbert} \| \nabla f(z)\|_F^2 \leq
\frac{2\kappa(1-\| f(z)\|^2_e)}{(1-\| z\|_X^2)^2},
\end{equation}
where $\kappa=\max\{ k_1, \cdots, k_p\}$.
\end{Cor}

\begin{proof}
 The inequality (\ref{product-ball-Schwarz-Pick}) and  the relation \eqref{Frobenius2}
imply \eqref{Frobenius-Schwarz-Pick-Hilbert}.
This completes the proof.
\end{proof}

The following theorem is a generalization  of
\eqref{1.6} to
pluriharmonic mappings from the unit ball $\B_X$ of a finite dimensional JB$^*$-triple $X$
to the Euclidean unit ball $\mathbb{B}^{n}$ of $\C^n$,
where $n=\dim X$.
Note that the condition
$\| \omega_f\|<1$ in $\B^{n}$
implies that
$f$ is a sense-preserving and locally univalent mapping in $\mathbb{B}^{n}$
(see \cite[p.637]{HK15}).
Also, the condition \eqref{P-Schwarz-Pick-quasiregular-Hilbert} is also an improvement of \eqref{1.6}.

\begin{Thm}
\label{Schwarz-quasiregular-Hilbert} Let $\B_X$ be a bounded
symmetric domain realized as the unit ball of a JB$^*$-triple $X$
with $\dim X=n<\infty$. Also, let $f=h+\overline{g}:\, \B_X\to
\mathbb{B}^{n}$ be a pluriharmonic mapping, where $h$ and
$g$ are holomorphic mappings of $\B_X$ into $\C^n$. Assume that $h$
is locally biholomorphic in $\B_X$ and $\| \omega_f\|\leq
k<1$ holds in $\B_X$. Then, for all $z\in\B_X$,
\begin{equation}
\label{P-Schwarz-Pick-quasiregular-Hilbert}
\| Dh(z)\|+\| Dg(z)\|
\leq
\frac{2K}{\sqrt{2(K^2+1)}}\frac{\sqrt{1-\| f(z)\|^2_e}}{1-\| z\|_X^2},
\end{equation}
where $K=(1+k)/(1-k)$.
\end{Thm}

The proofs of Theorems \ref{Har-1}$\sim$\ref{Schwarz-quasiregular-Hilbert} will be given in part II of Section \ref{sec-4}.

\section{Proofs of the main results}\label{sec-4}
\subsection*{Part I} Schwarz  type lemmas of holomorphic mappings and their applications.
\subsection*{ The proof of Theorem \ref{thm-h1}} Let $z$ be any fixed point in $B_{X}\backslash\{0\}$. Without
loss of generality, we assume  that $f(z)\neq0$. Let
$\eta=z/\|z\|_{X}\in\partial B_{X}$. For any fixed $b\in
\partial B_{Y}$, let

$$F(\xi)=l_{b}(f(\xi\eta)),~\xi\in\U,$$
where $l_{b}\in T(b)$. Then $F$ is a holomorphic mapping of $\U$
into $\overline{\U}$. By Remark \ref{cor-Kr-1}, we have

\be\label{eq-u1}|F(\xi)|\leq\frac{|F(0)|+|\xi|}{1+|F(0)|\,|\xi|}\ee
for $\xi\in\U$. Elementary calculation leads to

\be\label{eq-u2}|F(0)|=|l_{b}(f(0))|\leq\|l_{b}\|_{Y^{\ast}}\|f(0)\|_{Y}=\|f(0)\|_{Y}.\ee

Since, for any fixed $a\in[0,1)$, the function
$\varrho(x)=(a+x)/(1+ax)$ is increasing with respect to the variable
$x\in[0,\infty)$, by (\ref{eq-u1}) and (\ref{eq-u2}), we see that

\be\label{eq-u3}|l_{b}(f(\xi\eta))|\leq\frac{\|f(0)\|_{Y}+|\xi|}{1+\|f(0)\|_{Y}|\xi|}.\ee
Finally, by letting $\xi=\|z\|_{X}$ and $b=f(z)/\|f(z)\|_{Y}$ in
(\ref{eq-u3}), we obtain

$$\|f(z)\|_{Y}\leq\frac{\|f(0)\|_{Y}+\|z\|_{X}}{1+\|f(0)\|_{Y}\|z\|_{X}}.$$

Next, we show the sharpness part. If there is a point $z_{0}\in B_X$ such that $\|f(z_{0})\|_{Y}=1$, then $\|f(z)\|_{Y}=1$ for all $z\in B_X$. In this case, the sharpness part is obvious. In the following, we assume that $\|f(z)\|_{Y}<1$  for all $z\in B_X$. For any fixed point $z_{0}\in B_{X}\backslash\{0\}$, let
 $l_{w_{0}}\in T(w_{0})$ be fixed, where
 $w_{0}=z_{0}/\|z_{0}\|_{X}$. For each fixed  $b\in\partial B_{Y}$ and any fixed $a\in[0,1)$, let
 $$f(z)=\frac{a+l_{w_{0}}(z)}{1+al_{w_{0}}(z)}b,~z\in B_{X}.$$
Then $f$ is a holomorphic mapping from $B_X$ into $B_Y$ and by taking $z=z_{0}$, we obtain
$$\|f(z_{0})\|_{Y}=\left|\frac{a+l_{w_{0}}(z_{0})}{1+al_{w_{0}}(z_{0})}\right|
=\frac{\| f(0)\|_Y+\|z_{0}\|_{X}}{1+\| f(0)\|_Y\|z_{0}\|_{X}},$$ which completes the
proof.
 \qed

\subsection*{ The proof of Theorem \ref{thm-h2}} Let $z$ be any fixed point in $B_{X}\backslash\{0\}$. Without
loss of generality, we assume  that $f(z)\neq0$. Let
$\eta=z/\|z\|_{X}\in\partial B_{X}$. For any fixed $b\in
\partial B_{Y}$, let

$$F(\xi)=l_{b}(f(\xi\eta)),~\xi\in\U,$$
where $l_{b}\in T(b)$. Then $F$ is a holomorphic mapping of $\U$
into itself with $F(0)=0$. By Theorem B, we have

\be\label{eq-u-1}
|F(\xi)|\leq\frac{|F'(0)|+|\xi|}{1+|F'(0)|\,|\xi|}|\xi| \quad \mbox{for $\xi\in\U$.}
\ee
 Elementary calculation leads to

\beqq
|F'(0)|=|l_{b}(Df(0)\eta)|\leq\|l_{b}\|_{Y^{\ast}}\|Df(0)\eta\|_{Y}\leq\|Df(0)\|,\eeqq
 which together with (\ref{eq-u-1}) yields that

\be\label{eq-u-3}|l_{b}(f(\xi\eta))|\leq\frac{\|Df(0)\|+|\xi|}{1+\|Df(0)\|\,|\xi|}|\xi|.\ee
Finally, by letting $\xi=\|z\|_{X}$ and $b=f(z)/\|f(z)\|_{Y}$ in
(\ref{eq-u-3}), we obtain

$$\|f(z)\|_{Y}\leq\frac{\|Df(0)\|+\|z\|_{X}}{1+\|Df(0)\|\|z\|_{X}}\|z\|_{X}.$$

Next, we show the sharpness part. For any fixed point $z_{0}\in
B_{X}\backslash\{0\}$, let
 $l_{w_{0}}\in T(w_{0})$ be fixed, where
 $w_{0}=z_{0}/\|z_{0}\|_{X}$. For any fixed  $b\in\partial B_{Y}$ and any fixed $a\in[0,1]$, let
 $$f(z)=l_{w_{0}}(z)\frac{a+l_{w_{0}}(z)}{1+al_{w_{0}}(z)}b,~z\in B_{X}.$$
Then $f$ is a holomorphic mapping from $B_X$ into $B_Y$ with $f(0)=0$
and by letting $z=z_{0}$, we have
$$\|f(z_{0})\|_{Y}=|l_{w_{0}}(z_{0})|\left|\frac{a+l_{w_{0}}(z_{0})}{1+al_{w_{0}}(z_{0})}\right|
=\|z_{0}\|_{X}\frac{\| Df(0)\|+\|z_{0}\|_{X}}{1+\| Df(0)\|\|z_{0}\|_{X}},$$ which
completes the proof.
 \qed

\subsection*{ The proof of Theorem \ref{thm-Bloch}}
 We only need to prove (\ref{eq-thm-3.0}) for $k\in(0,1)$, since $\mathscr{B}_{f}/\mathscr{B}_{h}\equiv1$ for $k=0$.
In the following, we assume that $k\in(0,1)$ and we divide the proof
into five steps.

$\mathbf{Step~ 1.}$ For
$f=h+\overline{g}\in\mathscr{PH}(k)$, we claim that $\mathscr{B}_{f}(0)=d(0,\partial f(\mathbf{B}))$ and $\mathscr{B}_{h}(0)=d(0,\partial h(\mathbf{B}))$,
where $d(0,\partial f(\mathbf{B}))$ and $d(0,\partial h(\mathbf{B}))$ denote the Euclidean distances from $0$ to $\partial f(\mathbf{B})$
and $\partial h(\mathbf{B})$, respectively. We only need to prove $\mathscr{B}_{f}(0)=d(0,\partial f(\mathbf{B}))$ because the proof of another one is similar. From the definition of $\mathscr{B}_{f}(0)$,
we see that $\mathscr{B}_{f}(0)$ is equal either to the Euclidean distance from $f(0)$ to a boundary
point of $f(\mathbf{B})$ or to the Euclidean distance from $f(0)$ to a critical value of $f$.
In the following, we will show that the critical value of $f$ does not exist.  Since $\|\omega_{f}\|\leq k$, we see that
$$\det J_{f}=|\det
Dh|^{2}\det\left(I-Dg[Dh]^{-1}\overline{Dg[Dh]^{-1}}\right)\neq0.$$ Consequently, $f$ is locally univalent in $\mathbf{B}$. Let $V$ be a subdomain of $\mathbf{B}$ such that $f(V)=\{w:\, \|w\|_{e}<\mathscr{B}_{f}(0)\}$. If there exist $z_{0}\in\partial V$ and $z_{1}\in V$ such that $f(z_{0})=f(z_{1})$, then, by the condition that $\det J_{f}\neq0$ in $\mathbf{B}$, there exist neighbourhoods $U_{0}(z_{0})$ of $z_{0}$, $U_{1}(z_{1})$ of $z_{1}$, and $U_{2}(f(z_{0}))$ of $f(z_{0})=f(z_{1})$ such that $f$ maps $U_{0}(z_{0})$ and $U_{1}(z_{1})$ onto $U_{2}(f(z_{0}))$ univalently, respectively. This contradicts the fact that $f$ is univalent in $V$.
Therefore, the critical points do not exist.
Hence $\mathscr{B}_{f}(0)=d(0,\partial f(\mathbf{B}))$.
 Without loss of generality,
we assume that there exists
a boundary point $\xi_{0}$ of $f(\mathbf{B})$ such that
$\xi_{0}\in\{w\in\mathbb{C}^{n}:\, \|w\|_{e}=\mathscr{B}_{f}(0)\}$. Let
$\ell_{\xi_{0}}=f^{-1}([0,\xi_{0}))$ be the preimage of the
semi-open segment $[0,\xi_{0})$ with the starting point $0$ in the
ball $\mathbf{B}$. Then
\begin{equation}
\label{eq-bf0}
\mathscr{B}_{f}(0)=\|\xi_{0}\|_{e}=\int_{\ell_{\xi_{0}}}\left\|df(z)\right\|_{e}=\inf_{\gamma}\int_{\gamma}\left\|df(z)\right\|_{e},
\end{equation}
where the minimum is taken over all smooth paths $[0,1)\ni
t\mapsto\gamma(t)\in\mathbf{B}$ with $\gamma(0)=0$ and
$\lim_{t\mapsto 1^{-}}\|\gamma(t)\|_{X}=1$. Similarly,
we  can also assume that
$$\mathscr{B}_{h}(0)=\|\xi_{1}\|_{e}=\int_{\ell_{\xi_{1}}}\left\|dh(z)\right\|_{e}=\inf_{\gamma}\int_{\gamma}\left\|dh(z)\right\|_{e},
$$
where $\xi_{1}$ is a boundary point of $h(\mathbf{B})$ such that
$\xi_{1}\in\{w\in\mathbb{C}^{n}:\, \|w\|_{e}=\mathscr{B}_{h}(0)\}$
and
the simple smooth curve $\ell_{\xi_{1}}=h^{-1}([0,\xi_{1}))$
is the preimage of the semi-open segment $[0,\xi_{1})$ under the
mapping $h$. For $t\in[0,1)$, let
$\ell_{\xi_{0}}:=\ell_{\xi_{0}}(t)=f^{-1}(\xi_{0}t)$ and
$\ell_{\xi_{1}}:=\ell_{\xi_{1}}(t)=h^{-1}(\xi_{1}t)$.

$\mathbf{Step~ 2.}$ We first  establish the lower bound of
$\mathscr{B}_{f}(0)/\mathscr{B}_{h}(0).$
Differentiation of the
equation $f^{-1}(f(z))=z$ yields the following two formulas:
$$Df^{-1}Dh+\overline{D}f^{-1}Dg=I$$
and
$$Df^{-1}\overline{Dg}+\overline{D}f^{-1}\overline{Dh}=O,
$$
which imply that \be\label{eq-r1}
Df^{-1}=[Dh]^{-1}\left(I-\overline{\omega}_{f}\omega_{f}\right)^{-1}
\ee and
\be\label{eq-r2}\overline{D}f^{-1}=-[Dh]^{-1}\left(I-\overline{\omega}_{f}\omega_{f}\right)^{-1}
\overline{\omega_{f}}, \ee where
$f^{-1}(w)=(\sigma_{1}(w),\ldots,\sigma_{n}(w))'$ and
$\overline{D}f^{-1}(w)=\left(\frac{\partial \sigma_{j}}{\partial
\overline{w}_{k}}\right)_{n\times n}$ for $k,j\in\{1,\ldots,n\}$.
Then, by (\ref{eq-r1}) and (\ref{eq-r2}), we have

\beqq
\left\|DhDf^{-1}\frac{\xi_{0}}{\|\xi_{0}\|_{e}}\right\|_{e}+\left\|Dh\overline{D}f^{-1}\frac{\overline{\xi_{0}}}{\|\xi_{0}\|_{e}}\right\|_{e}&\leq&
\left\|\left(I-\overline{\omega}_{f}\omega_{f}\right)^{-1}\right\|\\
&&+\left\|\left(I-\overline{\omega}_{f}\omega_{f}\right)^{-1}
\overline{\omega}_{f}\right\|\\
&\leq&\frac{1}{1-\|\omega_{f}\|^{2}}+\frac{\|\omega_{f}\|}{1-\|\omega_{f}\|^{2}}\\
&=&\frac{1}{1-\|\omega_{f}\|}. \eeqq
 Consequently,

 \beqq
\mathscr{B}_{h}(0)&=&\int_{0}^{1}\left\|dh(\ell_{\xi_{1}}(t))\right\|_{e}\leq\int_{0}^{1}\left\|dh(\ell_{\xi_{0}}(t))\right\|_{e}\\
&=&\int_{0}^{1}\left\|\left(Dh(\ell_{\xi_{0}}(t))Df^{-1}(\xi_{0}t)\xi_{0}+
Dh(\ell_{\xi_{0}}(t))\overline{D}f^{-1}(\xi_{0}t)\overline{\xi_{0}}\right)dt\right\|_{e}\\
&\leq&\|\xi_{0}\|_{e}\int_{0}^{1}\bigg(\left\|Dh(\ell_{\xi_{0}}(t))Df^{-1}(\xi_{0}t)\frac{\xi_{0}}{\|\xi_{0}\|_{e}}\right\|_{e}\\
&&+
\left\|Dh(\ell_{\xi_{0}}(t))\overline{D}f^{-1}(\xi_{0}t)\frac{\overline{\xi_{0}}}{\|\xi_{0}\|_{e}}\right\|_{e}\bigg)dt\\
&\leq&\mathscr{B}_{f}(0)\int_{0}^{1}\frac{dt}{1-\|\omega_{f}(\ell_{\xi_{0}}(t))\|},
\eeqq which gives that

\be\label{At-1}
\frac{\mathscr{B}_{f}(0)}{\mathscr{B}_{h}(0)}\geq\frac{1}{\int_{0}^{1}\frac{dt}{1-\|\omega_{f}(\ell_{\xi_{0}}(t))\|}}\geq1-k.
\ee

$\mathbf{Step~ 3.}$ In this step,  we will give the lower bound of
$\mathscr{B}_{f}(z)/\mathscr{B}_{h}(z)$ for all $z\in \mathbf{B}$.
Since $ \mathbf{B}$ is homogeneous, we see that,
 for any fixed $\zeta\in\mathbf{B}$, there exists a
$\phi\in\mbox{Aut}(\mathbf{B})$ such that $\phi(0)=\zeta.$ For $z\in
\mathbf{B}$, let

\be\label{eq-j1} F(z)=f(\phi(z))-f(\phi(0))
 =H(z)+\overline{G(z)},
\ee
 where $H(z)=h(\phi(z))-h(\phi(0))$ and $G(z)=g(\phi(z))-g(\phi(0))$. Then $H(0)=G(0)=0$ and
$$\|\omega_{F}(z)\|=\|DG(z)[DH(z)]^{-1}\|=\|\omega_{f}(\phi(z))\|\leq k,
$$
which imply that $F\in\mathscr{PH}(k).$

By (\ref{At-1}) and (\ref{eq-j1}), we have \be\label{eq-l2}
\mathscr{B}_{F}(0)=\mathscr{B}_{f}(\zeta)
\geq(1-k)\mathscr{B}_{H}(0). \ee  Note that
$
 \mathscr{B}_{H}(0)=\mathscr{B}_{h}(\zeta),
$ which, together with (\ref{eq-l2}), implies that
\be\label{sharp-1}
\mathscr{B}_{f}(\zeta)\geq(1-k)\mathscr{B}_{h}(\zeta). \ee

Next we prove that (\ref{sharp-1}) is sharp for all
$\zeta\in\mathbf{B}$. Let
\[
R=\inf\{\|z\|_{e}:\, z\in\partial\mathbf{B}\}.
\]
Then there exists a point $z_0\in \partial\mathbf{B}$ such that $\| z_0\|_e=R$.
Let $U$ be a unitary transformation of $\C^n$ such that
$Uz_0$ is a pure imaginary vector in $\C^n$.
For $z\in\mathbf{B}$, let
$$f(z)=h(z)+\overline{g(z)}=Uz+k\overline{Uz},
$$
where $k\in[0,1)$ is a constant. Then $\mathscr{B}_{h}(0)=R$.
Also, $f$ is univalent on $\mathbf{B}$ and $\mathscr{B}_{f}(0)=R(1-k)$, which
gives that $\mathscr{B}_{f}(0)/\mathscr{B}_{h}(0)=1-k.$ In the
following, we will show that
$\mathscr{B}_{f}(\zeta)/\mathscr{B}_{h}(\zeta)=1-k$ for all
$\zeta\in\mathbf{B}$.

For any fixed $\zeta\in\mathbf{B}$, let \beqq
F(z)=f(\phi(z))-f(\phi(0))  =H(z)+\overline{G(z)}, \eeqq where
$\phi\in\mbox{Aut}(\mathbf{B})$ with $\phi(0)=\zeta$. Then
$$\mathscr{B}_{f}(\zeta)=\mathscr{B}_{F}(0)=(1-k)\mathscr{B}_{H}(0)=(1-k)\mathscr{B}_{h}(\zeta).
$$

$\mathbf{Step ~4.}$ Now we estimate the upper bound of
$\mathscr{B}_{f}(0)/\mathscr{B}_{h}(0).$
From (\ref{eq-bf0})
and the relation
$Dh(\ell_{\xi_{1}}(t))\ell_{\xi_{1}}'(t)=\xi_{1}$,
we have
\beq\label{eq-cpv-1}
\mathscr{B}_{f}(0)&=&\int_{0}^{1}\left\|df(\ell_{\xi_{0}}(t))\right\|_{e}\leq
\int_{0}^{1}\left\|df(\ell_{\xi_{1}}(t))\right\|_{e} \nonumber \\
&=&\int_{0}^{1}\left\|\big(Dh(\ell_{\xi_{1}}(t))\ell_{\xi_{1}}'(t)+
\overline{Dg(\ell_{\xi_{1}}(t))\ell_{\xi_{1}}'(t)}\big)dt\right\|_{e}
\nonumber \\
&=&\int_{0}^{1}\left\|\left(\xi_{1}+\overline{Dg(\ell_{\xi_{1}}(t))[Dh(\ell_{\xi_{1}}(t))]^{-1}Dh(\ell_{\xi_{1}}(t))
\ell_{\xi_{1}}'(t)}\right)dt\right\|_{e} \nonumber \\
&=&\int_{0}^{1}\left\|\left(\xi_{1}+\overline{\omega_{f}(\ell_{\xi_{1}}(t))\xi_{1}}\right)dt\right\|_{e} \nonumber \\
&\leq&\|\xi_{1}\|_{e}\int_{0}^{1}\left(1+\|\omega_{f}(\ell_{\xi_{1}}(t))\|\right)dt \nonumber \\
&=&\mathscr{B}_{h}(0)\left(1+\int_{0}^{1}\|\omega_{f}(\ell_{\xi_{1}}(t))\|dt\right).
\eeq Applying Theorem \ref{thm-h1} to $\omega_{f}/k$, we have
\be\label{eq-cpv-2}
\frac{\|\omega_{f}(z)\|}{k}\leq\frac{\|z\|_{X}+\frac{\|\omega_{f}(0)\|}{k}}{1+\frac{\|\omega_{f}(0)\|}{k}\|z\|_{X}}
\ee for $z\in\mathbf{B}$, where $X=\mathbb{C}^{n}$. Since $h^{-1}(w\|\xi_{1}\|_e)$ biholomorphically
maps $\B^n$ onto some subdomain of $\mathbf{B}$ with
$h^{-1}(0)=0$, by Theorem \ref{thm-h2}, we see that
$\|\ell_{\xi_{1}}(t)\|_{X}\leq t$. Consequently, by (\ref{eq-cpv-1})
and (\ref{eq-cpv-2}), we have

\beq\label{jjk-1} \mathscr{B}_{f}(0)&\leq&
\mathscr{B}_{h}(0)\left(1+k\int_{0}^{1}\frac{\|\omega_{f}(0)\|+k\|\ell_{\xi_{1}}(t)\|_{X}}{k+\|\omega_{f}(0)\|\,\|\ell_{\xi_{1}}(t)\|_{X}}dt\right)  \\
&\leq&
\mathscr{B}_{h}(0)\mu_{k}\left(\frac{\|\omega_{f}(0)\|}{k}\right)\nonumber.
\eeq For any fixed $k\in[0,1)$, it is not difficult to see that
$\mu_{k}(x)$ is an increasing function of $x \in (0,1].$

 $\mathbf{Step~ 5.}$ At last, we will establish the upper bound of
$\mathscr{B}_{f}(z)/\mathscr{B}_{h}(z)$ for all $z\in \mathbf{B}$.
For any fixed $\zeta\in\mathbf{B}$, let $\phi\in\mbox{Aut}(\mathbf{B})$
with $\phi(0)=\zeta.$ It follows from (\ref{eq-j1}) and
(\ref{jjk-1}) that


\beq\label{eq-j2} \mathscr{B}_{f}(\zeta)=\mathscr{B}_{F}(0)
\leq\mu_{k}\left(\frac{\|\omega_{f}(\phi(0))\|}{k}\right)\mathscr{B}_{H}(0),
\eeq where $F$ and $H$ are defined  in (\ref{eq-j1}). Note that
\be\label{eq-j3} \mathscr{B}_{H}(0)=\mathscr{B}_{h}(\zeta). \ee It
follows from (\ref{eq-j2}) and (\ref{eq-j3}) that

\be\label{R-1}
\mathscr{B}_{f}(\zeta)\leq\mu_{k}\left(\frac{\|\omega_{f}(\zeta)\|}{k}\right)\mathscr{B}_{h}(\zeta)\leq\mu_{k}(1)\mathscr{B}_{h}(\zeta)=(1+k)\mathscr{B}_{h}(\zeta).
\ee
Furthermore, the estimate in (\ref{R-1}) is asymptotically sharp
because
$$\lim_{k\rightarrow0^{+}}\mu_{k}(x)=1.
$$
The proof of the theorem is finished. \qed

\subsection*{ The proof of Theorem \ref{thm-3h}} The triangle inequality leads to
\be\label{PV- Eq9} \left\|\frac{f(z)-f(b)}{\|z\|_{X} -
\|b\|_{X}}\right\|_{Y} \geq \frac{1-\|f(z)\|_{Y}}{1-\|z\|_{X}}\ee
for $z\in B_X$.
 It
follows from Theorem \ref{thm-h2} that

\beqq\label{PV- Eq 10} \frac{1-\|f(z)\|_{Y}}{1-\|z\|_{X}} \geq
\frac{1+\|z\|_{X}}{1+\|Df(0)\| \|z\|_{X}},
 \eeqq
 which, together with (\ref{PV- Eq9}), implies that
\beq\label{PV- Eq 11}
 \liminf\limits_{r\rightarrow 1^{-}} \left\|\frac{f(rb)-f(b)}
 {\|rb\|_{X} - \|b\|_{X}}\right\|_{Y} &\geq& \liminf\limits_{r\rightarrow 1^{-}} \frac{1+\|rb\|_{X}}{1+\|Df(0)\| \|rb\|_{X}}\\ \nonumber
  &=& \frac{2}{1+\|Df(0)\|}.
  \eeq
Since the radial derivative
\[
Df(b)b=\lim_{r\rightarrow 1^{-}}\frac{f(rb)-f(b)}{r-1}
\]
exists,
the desired result follows from \eqref{PV- Eq 11}.

For a given value $\| Df(0)\|=r\in [0,1]$, the sharpness part
follows from the mapping
\[
f(z)=\frac{l_b(z)+r}{1+rl_b(z)}l_b(z)y,
\quad
z\in B_X,
\]
where $l_b\in T(b),$ and $y\in \partial B_Y$ is arbitrary. \qed

\begin{ThmL}\label{Zhu-1.1}{\rm  (\cite[Theorem 1.1]{Z18})}
Let $f$ be a holomorphic self-mapping of $\mathbb{U}$. If $f$ is holomorphic at $z=1$ with $f(1)=1$,
then  $$f'(1)\geq\frac{2|1-f(0)|^{2}}{1-|f(0)|^{2}+|f'(0)|}.$$ This estimate is sharp
with equality possible for each value of $f(0)$
and $|f'(0)|$
with $|f'(0)|\leq 1-|f(0)|^2$.
The extreme function is
$$f(z)=\frac{\gamma\mathcal{A}(z)+f(0)}{1+\gamma\overline{f(0)}\mathcal{A}(z)},$$
where $\gamma=(1-f(0))/(\overline{1-f(0)})$ and $\mathcal{A}(z)=z((1-|f(0)|^{2})z+|f'(0)|)/((1-|f(0)|^{2})+|f'(0)|z).$
\end{ThmL}

\subsection*{ The proof of Theorem \ref{thm-Sy}}
For $\zeta\in\U$, let $$F(\zeta)=\frac{1}{2c(\B_Y)} h_0(f(\zeta \alpha),\beta),$$
where $f:\,G\to \B_Y$ is a holomorphic mapping with $f(\alpha)=\beta.$
By the assumption, and Loos \cite[Theorem 6.5]{L77}, we have $F(1)=1$.
Then $F$ is a holomorphic mapping of $\U$ into itself
such that $F$ is holomorphic at $\zeta=1$ and $F(1)=1$. Elementary computations lead to
$$F'(1)=\frac{1}{2c(\B_Y)} h_0(Df(\alpha)\alpha,\beta)~\mbox{and}~F'(0)=\frac{1}{2c(\B_Y)} h_0(Df(0)\alpha,\beta),$$
which, together with Theorem L, yield that
\beqq
\frac{1}{2c(\B_Y)}
h_0(Df(\alpha)\alpha, \beta)\geq
\frac{2\left| 1-\frac{1}{2c(\B_Y)}h_0(f(0),\beta)\right|^2}
{1-\left|\frac{1}{2c(\B_Y)}h_0(f(0),\beta)\right|^2+\| Df(0)\alpha\|_Y},
\eeqq
where we have used the inequality $|F'(0)|\leq \| Df(0)\alpha\|_Y$.
{Next, we prove the sharpness part.
Since $G=B_{X}$ is the unit ball of $X$,
for any holomorphic function $\varphi$ of $\U$ into itself and for any
$l_{\alpha}\in T(\alpha)$,
the mapping $f(z)=\varphi(l_{\alpha}(z))\beta$ is a holomorphic mapping of $B_{X}$ into $\B_Y$.
Then, it follows from Theorem L that
there exists a holomorphic mapping of $B_{X}$ into $\B_Y$  with $f(\alpha)=\beta$ such that $$ h_0(f(\zeta \alpha),\beta)=2c(\B_Y)\frac{a+\epsilon\tau(\zeta)}{1+\epsilon\overline{a}\tau(\zeta)},$$
where $$a=\frac{1}{2c(\B_Y)} h_0(f(0),\beta),~b=\frac{1}{2c(\B_Y)} h_0(Df(0)\alpha,\beta)$$
and $$
\epsilon=\frac{1-a}{1-\overline{a}},~
\tau(\zeta)=\zeta\frac{(1-|a|^{2})\zeta+|b|}{(1-|a|^{2})+|b|\zeta}.$$}
The proof of this theorem is finished.
\qed

\subsection*{Part II} Schwarz  type lemmas of pluriharmonic mappings and their applications.

\subsection*{ The proof of Theorem \ref{Har-1}} Let $z\in B_X\setminus\{ 0\}$ be fixed. Without loss of generality, we may assume that
$a=f(z)-\frac{1-\|z\|_X^{2}}{1+\|z\|_X^{2}}f(0)\neq 0$. Let $w=z/\|
z\|_X\in \partial B_X$ and let $u\in \partial B_{Y}$ be arbitrarily
fixed. Since, for each $l_u\in T(u)$,
\[
\varphi(\zeta)=l_u(f(\zeta w)) ,\quad \zeta \in \U,
\]
is a harmonic mapping in $\U$ such that $\varphi(\U)\subseteq \U$,
we obtain from \cite[Theorem 3.6.1]{Pav1} (or \cite[Theorem 1]{Het})  that, for all $\zeta\in \U$,
\[
\left|l_u(f(\zeta
w))-\frac{1-|\zeta|^{2}}{1+|\zeta|^{2}}l_u(f(0))\right|\leq
\frac{4}{\pi}{\arctan}| \zeta|.
\]
Especially, let $\zeta=\| z\|_X$. Since $\zeta w=z$, we have
\[
|l_u(a)|=\left|l_u\left(f(z)-\frac{1-\|z\|_X^{2}}{1+\|z\|_X^{2}}f(0)\right)\right|\leq
\frac{4}{\pi}{\arctan}\| z\|_X.
\]
Finally, if $u=a/\| a\|_{Y}$, then we obtain
\[
\left\|f(z)-\frac{1-\|z\|_X^{2}}{1+\|z\|_X^{2}}f(0)\right\|_{Y}\leq
\frac{4}{\pi}{\arctan}\| z\|_X.
\]

Next, we prove the sharpness part. For any fixed point $z_{0}\in B_{X}\backslash\{0\}$, let
 $l_{w_{0}}\in T(w_{0})$ be fixed, where $w_{0}=z_{0}/\|z_{0}\|_{X}$. It follows from
 \cite[Theorem 3.6.1]{Pav1} (or \cite[Theorem 1]{Het}) that there exists a harmonic mapping $\Phi$ of $\U$ into itself with $\Phi(0)=0$ such that
 \beqq \left|\Phi(l_{w_{0}}(z_{0}))\right|=\frac{4}{\pi}\arctan|l_{w_{0}}(z_{0})|.\eeqq
 For any fixed  $b\in\partial B_{Y}$, let
 $$f(z)=\Phi(l_{w_{0}}(z))b,~z\in B_{X}.$$
 Then
$f:\,B_X\to B_Y$ is harmonic and
\beqq
 \left\|f(z_{0})\right\|&=&
 \left|\Phi(l_{w_{0}}(z_{0}))\right|
 =\frac{4}{\pi}\arctan|l_{w_{0}}(z_{0})|\\
 &=&\frac{4}{\pi}\arctan\|z_{0}\|_{X},
 \eeqq
which completes the proof.
\qed

\begin{ThmM}\label{Zhu-1.7}{\rm  (\cite[Theorem 1.7]{Z18})}
Let $f$ be a harmonic mapping  of $\mathbb{U}$ into itself
with $f(0)=0$. Then, for $z\in\U$, we have
 $$|f(z)|\leq\frac{4}{\pi}\arctan\left(\frac{|z|+\frac{\pi}{4}(|f_{z}(0)|+
 |f_{\overline{z}}(0)|)}{1+\frac{\pi}{4}(|f_{z}(0)|+|f_{\overline{z}}(0)|)|z|}|z|\right).$$

\end{ThmM}

\subsection*{ The proof of Theorem \ref{Har-2}} For any fixed $z\in B_X\setminus\{ 0\}$, let $w=z/\|
z\|_X\in \partial B_X$. Without loss of generality, we assume $f(z)\neq0.$
Let $u\in \partial B_{Y}$ be arbitrarily
fixed. Since, for each $l_u\in T(u)$, the function $\varphi$ defined by
\[
\varphi(\zeta)=l_u(f(\zeta w)) 
\]
is a harmonic mapping of $\U$ into itself
with $\varphi(0)=0$,
we obtain from Theorem M  that, for all $\zeta\in \U$,

\begin{equation}\label{eq-jh-1.1}
|l_u(f(\zeta w))|=|\varphi(\zeta)|\leq\frac{4}{\pi}\arctan\left(\frac{|\zeta|+
\frac{\pi}{4}(|\varphi_{\zeta}(0)|+|\varphi_{\overline{\zeta}}(0)|)}{1+\frac{\pi}{4}(|\varphi_{\zeta}(0)|+|\varphi_{\overline{\zeta}}(0)|)|\zeta|}|\zeta|\right).
\end{equation}
By the definition of $\Lambda_f(0;w)$, we have
\[
|\varphi_{\zeta}(0)|+|\varphi_{\overline{\zeta}}(0)|
\leq \Lambda_f(0;w),
\]
which, together with (\ref{eq-jh-1.1}), implies that

\beq\label{eq-jh-1.2}
|l_u(f(\zeta w))|\leq\frac{4}{\pi}\arctan\left(\frac{|\zeta|+
\frac{\pi}{4}\Lambda_{f}(0;w)}{1+\frac{\pi}{4}\Lambda_{f}(0;w)|\zeta|}|\zeta|\right).
\eeq
By letting $\zeta=\|z\|_{X}$ in (\ref{eq-jh-1.2}),  we have

$$
|l_u(f(z))|\leq\frac{4}{\pi}\arctan\left(\frac{\|z\|_{X}+
\frac{\pi}{4}\Lambda_{f}(0;w)}{1+\frac{\pi}{4}\Lambda_{f}(0;w)\|z\|_{X}}\|z\|_{X}\right).
$$
Finally, if  $u=f(z)/\|f(z)\|_{Y}$, then we get the desired result.
\qed

\subsection*{ The proof of Theorem \ref{thm-Har3a}}

It
follows from Theorem \ref{Har-1} that
\beqq\label{PV- Eq 1000} \frac{1-\|f(z)\|_{Y}}{1-\|z\|_{X}} \geq
\frac{1-\frac{4}{\pi}\arctan\| z\|_X}{1-\|z\|_X}-\frac{1+\| z\|_X}{1+\| z\|_X^2}\| f(0)\|_Y.
 \eeqq
 By using arguments similar to those in the proof of Theorem \ref{thm-3h},
 we obtain

\begin{align}\label{eq-3.3a}
 \| Df(b)b\|_Y&\geq \liminf\limits_{r\rightarrow 1^{-}}
 \left(\frac{1-\frac{4}{\pi}\arctan\| rb\|_X}{1-\|rb\|_X}-\frac{1+\| rb\|_X}{1+\| rb\|_X^2}\| f(0)\|_Y\right)\\ \nonumber
 &=
 \frac{2}{\pi}-\| f(0)\|_Y.
 \end{align}
Next, let $l_{f(b)}\in T(f(b))$ and 
 \[
 p(\zeta)=1-{\rm Re}\big(l_{f(b)}(f(\zeta b))\big),
 \quad
 \zeta \in \U.
 \]
 Then $p$ is a positive real valued harmonic function on $\U$
 with $p(1)=0$.
 By Harnack's inequality, we have
 \[
 \frac{1-r}{1+r}p(0)\leq p(\zeta),
 \quad
 r=|\zeta|<1.
 \]
 Therefore, we have
 \[
 \frac{1}{1+r}p(0)\leq \frac{p(r)-p(1)}{1-r},
 \quad
 0<r<1.
 \]
 Letting $r\to 1^-$, we have
\begin{equation}
 \label{eq-3.3b}
 \frac{1-\| f(0)\|_Y}{2}\leq
 \frac{1-{\rm Re}\big(l_{f(b)}(f(0))\big)}{2}\leq
  {{\rm Re}\big(l_{f(b)}(Df(b)b)\big)}
 \leq \| Df(b)b\|_Y.
 \end{equation}
{Then combining (\ref{eq-3.3a}) and (\ref{eq-3.3b})
 gives the desired result.}
 This completes the proof.
  \qed

\subsection*{ The proof of Theorem \ref{Schwarz-pluriharmonic-symmetric}}

By the definition of $c(\B_Y)$ and the assumption on $f$, {we see that}
$\varphi$ is a harmonic mapping of $\U$ into itself with $\varphi(0)=0$.
Also, $\varphi$ is differentiable at $\zeta=1$
and $\varphi(1)=1$.
Then, by \cite[Theorem 1.8]{Z18},
we have
\begin{eqnarray*}
{\rm Re} \big(
\varphi_{\zeta}(1)+\varphi_{\overline{\zeta}}(1)
\big)
&\geq&
\frac{4}{\pi}
\frac{1}
{1+\frac{\pi}{4}(|\varphi_{\zeta}(0)|+|\varphi_{\overline{\zeta}}(0)|)}
\\
&\geq&
\frac{4}{\pi}
\frac{1}
{1+\frac{\pi}{4}\Lambda_f(0;\alpha)},
\end{eqnarray*}
which implies that (\ref{eq-Schwarz-pluriharmonic-symmetric}).
The mapping $f(z)=\psi(l_{\alpha}(z))\beta$
gives the sharpness, where
\[
\psi(\zeta)=\frac{2}{\pi}\arctan\frac{2{\rm Re}(\zeta)}{1-|\zeta|^{2}},
\quad
\zeta\in \U.
\]
This completes the proof.
\qed

\subsection*{ The proof of Theorem \ref{Schwarz-norm}}
First, we consider the case $z_0=0$.
By {Proposition \ref{Po-1.0}},
it suffices to show that
$$
\left\{
\begin{array}{cl}
\| Df(0)\|
\leq \dfrac{4}{\pi} & \mbox{if $f(0)=0$};\\
\sup_{\| \beta\|_{X}=1}\left|l_{f(0)}(Df(0)\beta)\right|
\leq \dfrac{4}{\pi} & \mbox{if $f(0)\neq 0$}.
\end{array}
\right.
$$

(i) If $f(0)=0$, then let $F(\zeta)=l(f(\zeta \beta))$ for $\zeta
\in \U$, where $\beta\in X$ with $\| \beta\|_X=1$ and $l\in Y^*$
with $\| l\|_{Y^*}=1$ are arbitrarily fixed. Then $F:\,\U \to \U$ is
harmonic. By applying Theorem F to the harmonic mapping $F$, we have
$$|l(Df(0)\beta)|\leq \frac{4}{\pi}.$$ Since $\beta\in X$ with $\|
\beta\|_X=1$ and $l\in Y^*$ with $\| l\|_{Y^*}=1$ are arbitrary, we
obtain that $$ \| Df(0)\|\leq \frac{4}{\pi}.$$

(ii)
If $f(0)\neq 0$,
then let $F(\zeta)=l_{f(0)}(f(\zeta \beta))$ for $\zeta \in \U$,
where
$\beta\in X$ with $\| \beta\|_X=1$ is arbitrarily fixed.
Then $F:\,\U \to \U$ is harmonic.
By applying Theorem F to the harmonic mapping $F$, we have
$$|l_{f(0)}(Df(0)\beta)|\leq \frac{4}{\pi}.$$
Therefore, we have proved (\ref{P-Schwarz-Pick-norm}) in the case $z_0=0$.

Next, we consider the case $z_0\neq 0$.
Let $g_{z_0}\in \mbox{Aut}(\B_X)$ be the
M\"{o}bius transformation of $\B_X$ defined by
(\ref{Mobius}).
By {Proposition \ref{Po-1.0},}
we have
\[
|\nabla \| f\|_Y(z_0)|
\leq
|\nabla \| f\circ g_{z_0}\|_Y(0)|\cdot
\| Dg_{z_0}(0)^{-1}\|_X.
\]
Since $f\circ g_{z_0}$ satisfies the assumptions of the theorem for $z_0=0$,
by
applying  (\ref{P-Schwarz-Pick-norm}) in the case $z_0=0$ and using (\ref{Kaup-estimate}),
we obtain that
\[
|\nabla \| f\|_Y(z_0)|
\leq
\frac{4}{\pi}\frac{1}{1-\| z_0\|_X^2}.
\]

Finally, we will show that the estimate
(\ref{P-Schwarz-Pick-norm}) is sharp.
Let $z_0\in \B_X\setminus \{0\}$ be fixed and
let $w_0=z_0/\| z_0\|_X$.
It follows from
 Theorem F that there exists a harmonic mapping $\phi$ of $\U$ into itself such that
$\phi(\| z_0\|_X)\in \mathbb{R}\setminus\{ 0 \}$ and
$$
|\phi_{\zeta}(\| z_0\|_X)|+|\phi_{\overline{\zeta}}(\| z_0\|_X)|=\frac{4}{\pi}\frac{1}{1-\| z_0\|_X^2}.
$$
For any fixed $l_{w_0}\in T(w_0)$ and any fixed
$a\in \partial B_Y$,
let
\[
f(z)=\phi(l_{w_0}(z))a,
\quad z\in \B_X.
\]
Then $f$ is a pluriharmonic mapping from $\B_X$ into $B_Y$.
Moreover,
\begin{align*}
|\nabla \| f\|_Y(z_0)|&=
\sup_{\| \beta\|_X=1}|\phi_{\zeta}(\| z_0\|_X)l_{w_0}(\beta)+\phi_{\overline{\zeta}}(\| z_0\|_X)\overline{l_{w_0}(\beta)}|
\\
&=
|\phi_{\zeta}(\| z_0\|_X)|+|\phi_{\overline{\zeta}}(\| z_0\|_X)|
\\
&=\frac{4}{\pi}\frac{1}{1-\| z_0\|_X^2}.
\end{align*}
If $z_0=0$, then for arbitrary $w_0\in {\partial \B}_X$, by using the above argument, we have
$$|\nabla \| f\|_Y(0)|=\frac{4}{\pi}.$$
This completes the proof.
\qed

\subsection*{ The proof of Theorem \ref{Schwarz-Hilbert}}
First, we show that
\begin{equation}
\label{eq4.2}
\sum_{j=1}^m \left(\| Dh(0)\widetilde{w}_j\|_e^2+\| Dg(0)\widetilde{w}_j\|_e^2\right)
\leq
{1-\| f(0)\|^2_e}
\end{equation}
for all $w_j\in X_j$ with $\| w_j\|_{X_j}=1$ $(1\leq j\leq m)$.
Indeed, let $w_j\in X_j$ with $\| w_j\|_{X_j}=1$ $(1\leq j\leq m)$
be fixed and let $F(\zeta_1,\dots, \zeta_m)=f(\zeta_1 w_1,
\dots,\zeta_m w_m)$ for $\zeta=(\zeta_1,\dots,\zeta_m)\in\U^m$.
Applying Theorem K to $F$, we obtain (\ref{eq4.2}).

Next, let $z\in \B_X\setminus \{ 0\}$ and $w=(w_1, \dots,w_m)\in X$
with $\| w_j\|_{X_j}=1$ for $j=1,\dots, m$ be fixed.
Let $g_z$ be the M\"{o}bius transformation
defined by
(\ref{Mobius}).
Applying (\ref{eq4.2}) to the pluriharmonic mapping
$f\circ g_z=h\circ g_z+\overline{g\circ g_z}$
and the unit vector $[Dg_z(0)]^{-1}w/\| [Dg_z(0)]^{-1}w\|_X$,
we have
\[
\frac{\sum_{j=1}^m \left(\| Dh(z)\widetilde{w}_j\|_e^2+\| Dg(z)\widetilde{w}_j\|_e^2\right)}{\| [Dg_z(0)]^{-1}w\|_X^2}\leq 1-\|f(z)\|^2_e.
\]
Therefore, by using (\ref{Kaup-estimate}), we have
\[
\sum_{j=1}^m \left(\| Dh(z)\widetilde{w}_j\|_e^2+\| Dg(z)\widetilde{w}_j\|_e^2\right)
\leq
\| [Dg_z(0)]^{-1}w\|_X^2( 1-\|f(z)\|^2_e)
\leq
\frac{1-\| f(z)\|^2_e}{(1-\| z\|_X^2)^2}
\]
as desired.

Next, we prove the sharpness part.
Let $a=(a_1,\dots, a_m)\in \B_X\setminus \{ 0\}$ with
$\| a_1\|_{X_1}=\cdots =\| a_m\|_{X_m}$
be arbitrarily fixed.
For $z\in\B_X$, let $$f(z)=(f_{1}(z),f_{2}(z),\ldots,f_{n}(z))\in
\B^n,$$
where
\[
f_{1}(z)=\frac{-\|a_{1}\|_{X_1}+l_a(z)}{1-\|a_{1}\|_{X_1}l_a(z)},
\]
and $f_{j}(z)\equiv0$ for $j\in\{2,\ldots,n \}$.
Then
this mapping gives the sharpness at $z=a$.
This completes the proof.
\qed
\subsection*{ The proof of Theorem \ref{Schwarz-quasiregular-Hilbert}}
Let $w\in X$ with $\| w\|_X=1$ be fixed. Since
\[
\| Dg(z)w\|_{e}=\| \omega_f(z)Dh(z)w\|_{e}\leq k\| Dh(z)w\|_{e} ~\mbox{ for all $z\in\B_X$,}
\]
there exists a function $\eta_{w}(z)\in [0,k]$ such that
$\|Dg(z)w\|_{e}=\eta_{w}(z)\| Dh(z)w\|_{e}$. Then, by
(\ref{P-Schwarz-Pick-Hilbert}) with $m=1$, we have

\[
\| Dh(z)w\|_{e} \leq
\frac{1}{\sqrt{1+\eta_{w}^2(z)}}\frac{\sqrt{1-\|
f(z)\|^2_e}}{1-\| z\|_X^2}
\]
and
\[
\| Dg(z)w\|_{e} \leq
\frac{\eta_{w}(z)}{\sqrt{1+\eta_{w}^2(z)}}\frac{\sqrt{1-\|
f(z)\|^2_e}}{1-\| z\|_X^2},
\]
which, together with the monotonicity of function
$\chi(t)=(1+t)/\sqrt{1+t^{2}}$ for $t\in[0,1)$, give that
\begin{align*}
\| Dh(z)\|+\| Dg(z)\| &\leq
\frac{1+\eta_{w}(z)}{\sqrt{1+\eta_{w}^2(z)}}\frac{\sqrt{1-\|
f(z)\|^2_e}}{1-\| z\|_X^2}
\\
&\leq
\frac{1+k}{\sqrt{1+k^2}}\frac{\sqrt{1-\| f(z)\|^2_e}}{1-\| z\|_X^2}
\\
&=
\frac{2K}{\sqrt{2(K^2+1)}}\frac{\sqrt{1-\| f(z)\|^2_e}}{1-\| z\|_X^2}.
\end{align*}
This completes the proof.
\qed

\section{A concluding remark}\label{s5}

Let $\B_j$ be the unit ball of a complex Hilbert space $H_j$ for $j=1,2$, respectively.
Note that if $f$ is $C^1$ at $z_0\in \partial \B_1$ with values in $H_2$,
then the adjoint operator $Df(z_0)^*$ is defined
by
\[
{\rm Re} \left(\langle Df(z_0)^*w,z\rangle_{H_1}\right)={\rm Re} \left(\langle w, Df(z_0)z\rangle_{H_2}\right)
\quad
\mbox{for }
z\in H_1,\, w\in H_2,\]
where $\langle\cdot,\cdot\rangle_{H_j}$ is the inner product of $H_j$, $j=1,2$.
The following result was obtained in \cite[Proposition 1.8]{GHK-JAM}.

\begin{Pro}
\label{p.schwarz-ph-ball}
Let $\B_j$ be the unit ball of a complex Hilbert space $H_j$ for $j=1,2$, respectively.
Let
$f:\,\B_1\to \B_2$ be a pluriharmonic mapping. Assume that $f$ is of class $C^1$ at some  point $z_0\in {\partial \B}_1$
and $f(z_0)=w_0\in {\partial \B}_2$. Then
there exists a constant $\lambda\in {\mathbb R}$ such that $Df(z_0)^*w_0=\lambda z_0$. Moreover,
$\lambda\geq\frac{1-{\rm Re}\big(\langle f(0),w_0\rangle\big)}{2}>0$.
\end{Pro}

By using Proposition \ref{p.schwarz-ph-ball} and the arguments similar to those in the proof of Theorem  \ref{thm-Har3a},
we obtain a better estimate:
\[
\lambda\geq \max\left\{ \frac{2}{\pi}-\| f(0)\|_{H_2},\frac{1-{\rm Re}\big(\langle f(0),w_0\rangle\big)}{2}\right\}.
\]

\bigskip



\section{Acknowledgments}
 We are grateful to the referee for her/his useful comments and suggestions.
The research of the first author was partly supported by the National Science
Foundation of China (grant no. 12071116), the Hunan Provincial Natural Science Foundation of China (No. 2022JJ10001),
 the Key Projects of Hunan Provincial Department of Education (grant no. 21A0429);
 the Double First-Class University Project of Hunan Province
(Xiangjiaotong [2018]469),  the Science and Technology Plan Project of Hunan
Province (2016TP1020), and the Discipline Special Research Projects of Hengyang Normal University (XKZX21002);
 The second author (Hidetaka Hamada)
 was partially supported by JSPS KAKENHI Grant
Number JP19K03553.


\begin{thebibliography}{1}

\bibitem{Ah}  L. V. Ahlfors,
An extension of Schwarz's lemma, \textit{Trans. Amer. Math. Soc.,}
{\bf 43} (1938), 359--364.

\bibitem{BGM14}
O. Blasco, P. Galindo, A. Miralles, Bloch functions on the unit
ball of an infinite dimensional Hilbert space, \textit{ J. Funct.
Anal.}, {\bf 267} (2014), 1188--1204.

\bibitem{B}  M. Bonk, On Bloch's constant, \textit{Proc. Amer. Math. Soc.}, {\bf 378} (1990), 889--894.

\bibitem{BE}  M. Bonk and A. Eremenko, Covering properties of meromorphic functions, negative curvature and spherical
geometry, \textit{Ann. Math.,} {\bf 152} (2000), 551--592.



\bibitem{BK}  D. M. Burns and S. G. Krantz,
Rigidity of holomorphic mappings and a new Schwarz lemma at the
boundary, \textit{J. Amer. Math. Soc.,} {\bf 7} (1994), 661--676.



\bibitem{CG01}
H. H. Chen, P. Gauthier, Bloch constants in several variables,
\textit{ Trans. Amer. Math. Soc.}, {\bf 353} (2001), 1371--1386.

\bibitem{C13}
H. H. Chen,
The Schwarz-Pick lemma and Julia lemma for real planar harmonic mappings, \textit{Sci. China Math.}, {\bf 56} (2013), 2327--2334.

\bibitem{CG11}
H. H. Chen, P. Gauthier, The Landau theorem and Bloch theorem for
planar harmonic and pluriharmonic mappings, \textit{Proc. Amer.
Math. Soc.}, {\bf 139} (2011), 583--595.




\bibitem{CH20}
S. L. Chen and H. Hamada,
Some sharp Schwarz-Pick type estimates and their applications of harmonic and pluriharmonic functions,
 \textit{J. Funct. Anal.,} {\bf 282} (2022), No. 1, Article ID 109254, 42 p.



\bibitem{CLW}  S. L. Chen, P. Li and X. T. Wang,
Schwarz-type lemma, Landau-type theorem, and Lipschitz-type space of solutions to inhomogeneous biharmonic equations, \textit{J. Geom. Anal.,} {\bf 29} (2019), 2469--2491.

\bibitem{CP-2017}  S. L. Chen and S. Ponnusamy,
Distortion and covering theorems of pluriharmonic mappings, \textit{Filomat,} {\bf 31} (2017), 2749--2762.

\bibitem{CPW-2021}  S. L.  Chen,  S. Ponnusamy and X. Wang,
Remarks on `Norm estimates of the partial derivatives for harmonic  mappings and harmonic quasiregular mappings', \textit{J. Geom. Anal.}  {\bf 31} (2021), 11051--11060.


\bibitem{C12}
C.-H. Chu,
{\it Jordan Structures in Geometry and Analysis},
in: Cambridge Tracts in Mathematics, vol. {\bf 190}, Cambridge University Press, Cambridge, 2012.


\bibitem{C-2021}
C.-H. Chu,
{\it Bounded symmetric domains in Banach spaces},
World Scientific Publishing Co. Pte., Hackensack, NJ, (2021), 393pp.


\bibitem{CHHK16}
C.-H. Chu, H. Hamada, T. Honda and G. Kohr,
Distortion of locally biholomorphic Bloch mappings on bounded symmetric domains, \textit{ J. Math. Anal. Appl.}, {\bf 441}  (2016), 830--843.

\bibitem{CHHK17}
C.-H. Chu, H. Hamada, T. Honda and G. Kohr, Bloch functions on
bounded symmetric domains, \textit{ J. Funct. Anal.}, {\bf 272}  (2017), 2412--2441.



\bibitem{C89}
F. Colonna, The Bloch constant of bounded harmonic mappings,
\textit{ Indiana Univ. Math. J.}, {\bf 38} (1989), 829--840.

\bibitem{Du}   P. Duren,
{\it Harmonic mappings in the plane,} Cambridge Univ. Press, 2004.

\bibitem{DHK-2011}  P. Duren, H. Hamada and G. Kohr, Two-point distortion theorems for harmonic and pluriharmonic mappings,
 \textit{Trans. Amer. Math. Soc.}, {\bf 363} (2011), 6197--6218.

\bibitem{EJLS}  M. Elin, F. Jacobzon, M. Levenshtein and D. Shoikhet, The Schwarz Lemma: Rigidity and Dynamics, in:
\textit{Harmonic and Complex Analysis and Applications,} Birkh\"user Basel, 2014, 135-230.


\bibitem{E-1} M. Elin, M. Levenshtein, S. Reich and D. Shoikhet, A rigidity
theorem for holomorphic generators on the Hilbert ball,
\textit{Proc. Amer. Math. Soc.}, {\bf 136} (2008), 4313--4320.

\bibitem{E-2} M. Elin, S. Reich and D. Shoikhet, A Julia-Carath\'{e}odory
theorem for hyperbolically monotone mappings in the Hilbert ball,
\textit{Israel J. Math.}, {\bf 164} (2008), 397--411.

\bibitem{FG}  C. H. Fitzgerald and S. Gong, The Bloch theorem in several complex variables, \textit{J. Geom. Anal.}, {\bf 4} (1996), 35--58.


\bibitem{GPS}  S. Y. Graf, S. Ponnusamy and V. V. Starkov, Radii of
covering disks for locally univalent harmonic mappings,
\textit{Monatsh. Math.}, {\bf 180} (2016), 527--548.

\bibitem{GHK-JAM}
I. Graham, H. Hamada and G. Kohr,
A Schwarz lemma at the boundary on complex Hilbert balls and
applications to starlike mappings
\textit{J. Anal. Math.}, \textbf{140} (2020), 31--53.

\bibitem{GV}  I. Graham and D. Varolin, Bloch constants in one and several variables, \textit{Pacific J. Math.,} {\bf 174} (1996), 347--357.



\bibitem{H18}
H. Hamada,
A Schwarz lemma at the boundary using the Julia-Wolff-Carath\'eodory type condition on finite dimensional irreducible bounded symmetric domains, \textit{J. Math. Anal. Appl.}, \textbf{465}  (2018), 196--210.

\bibitem{H19JAM}
H. Hamada,
A distortion theorem and the Bloch constant for Bloch mappings  in $\mathbb{C}^n$,
\textit{J. Anal. Math.},
\textbf{137} (2019), 663--677.


\bibitem{HHK}
H. Hamada, T. Honda and G. Kohr,
Trace-order and a distortion theorem for linearly invariant families on the unit ball of a
finite dimensional JB$^*$-triple,
\textit{J. Math. Anal. Appl.},
\textbf{396} (2012), 829--843.


\bibitem{HK15}
H. Hamada and G. Kohr, Pluriharmonic mappings in $\mathbb{C}^n$ and
complex Banach spaces, \textit{ J. Math. Anal. Appl.}, \textbf{426}  (2015), 635--658.


\bibitem{HKo}
H. Hamada and G. Kohr,
A rigidity theorem at the boundary for holomorphic mappings with values in finite dimensional bounded symmetric domains,
\textit{Math. Nachr.},
{\bf 294} (2021), 2151--2159.

\bibitem{H-1977}
L. A. Harris, On the size of balls covered by analytic transformations, \textit{ Monatsh. Math.}, \textbf{83} (1977), 9--23.

\bibitem{HRS}  L. A. Harris, S. Reich and D. Shoikhet, Dissipative holomorphic
functions, Bloch radii, and the Schwarz lemma, \textit{J. Anal.
Math.}, {\bf 82} (2000), 221--232.


\bibitem{Hei}
E. Heinz, On one-to-one harmonic mappings, \textit{ Pacific J.
Math.}, \textbf{9} (1959), 101--105.


\bibitem{H-1940} A. Herzig,
Die Winkelderivierte und das Poisson-Stieltjes-Integral,
\textit{Math. Z.}, \textbf{46} (1940), 129--156.

\bibitem{Het}   H. W. Hethcote, Schwarz lemma analogues for harmonic functions,
\textit{Int. J. Math. Educ. Sci. Technol.}, {\bf 8} (1977), 65--67.

\bibitem{Hu}  X. J. Huang,
A preservation principle of extremal mappings near a strongly
pseudoconvex point and its applications, \textit{Ill. J. Math.,}
{\bf 38} (1994), 283--302.

\bibitem{Hu-1}  X. J. Huang,
A boundary rigidity problem for holomorphic mappings on some weakly
 pseudoconvex domains, \textit{Can. J.
Math.,} {\bf 47} (1995), 405--420.

\bibitem{Iz}  A. J. Izzo,
Uniform algebras generated by holomorphic and pluriharmonic
functions, \textit{Trans. Amer. Math. Soc.,} {\bf 339} (1993),
835--847.


\bibitem{KV12}
D. Kalaj and M. Vuorinen, On harmonic functions and the Schwarz
lemma, \textit{Proc. Amer. Math. Soc.}, \textbf{140} (2012),
161--165.

\bibitem{K94}
W. Kaup, {Hermitian Jordan triple systems and the automorphisms of
bounded symmetric domains}, \textit{ Math. Appl.}, {\bf 303} (1994)
204--214, Kluwer Acad. Publ., Dordrecht.



\bibitem{Kra}  S. G. Krantz,
{\it Geometric function theory:} Explorations in complex analysis,
Birkh\"auser Boston 2006.

\bibitem{Lan} E. Landau, \"Uber die Bloch'sche Konstante und zweiverwandte Weltkonstanten,
\textit{Math. Z.,}  {\bf 30} (1929), 608--634.

\bibitem{Lelong89}
P. Lelong,
Fonction de Green pluricomplexe et lemmes de Schwarz dans les espaces de Banach,
\textit{J. Math. Pures Appl.} (9),
{\bf 68} (1989), 319--347.

\bibitem{Liu-21}    C. W. Liu, A proof of the Khavinson conjecture, \textit{Math. Ann.,} {\bf 380} (2021), 719--732.


\bibitem{L-T-2015}
 T. Liu and X. Tang, A new boundary rigidity theorem for holomorphic self-mappings of the unit ball in $\mathbb{C}^{n}$, \textit{Pure Appl. Math. Q.}, {\bf 11} (2015), 115--130.

\bibitem{L-T}
 T. Liu and X. Tang, Schwarz lemma at the boundary of strongly
pseudoconvex domain in $C^n$, \textit{ Math. Ann.}, {\bf 366} (2016), 655--666.

\bibitem{L-T-2020}
 T. Liu and X. Tang, Schwarz lemma and rigidity theorem for holomorphic mappings on the unit polydisk in $\mathbb{C}^{n}$, \textit{J. Math. Anal. Appl.}, {\bf 489} (2020), 1--9.


\bibitem{LM}X. Y. Liu and C. D. Minda, Distortion theorems for Bloch functions, \textit{Trans. Amer. Math. Soc.},
{\bf 333} (1992), 325--338.


\bibitem{L77}
O. Loos,
Bounded symmetric domains and Jordan pairs,
University of California, Irvine, 1977.


\bibitem{Mo-1} N. Mok,  Nonexistence of proper holomorphic maps between certain classical bounded symmetric
domains, \textit{Chin. Ann. Math. Ser. B}, {\bf 29} (2008), 135--146.

\bibitem{Mo-2} N. Mok,  Holomorphic isometries of the complex unit ball into irreducible bounded symmetric domains.
\textit{Proc. Am. Math. Soc.},  {\bf 144} (2016), 4515--4525.

\bibitem{Mo-3}  N. Mok and S. Ng,  Germs of measure-preserving holomorphic maps from bounded symmetric domains
to their Cartesian products, \textit{J. Reine Angew. Math.}, {\bf 669} (2012), 47--73.




\bibitem{Os}  R. Osserman,
A sharp Schwarz inequality on the boundary, \textit{Proc. Amer. Math.
Soc.,} {\bf 128} (2000), 3513--3517.

\bibitem{Pav1}   M. Pavlovi\'c,  {\it Introduction to function spaces on the disk,}
Matemati$\breve{\mbox{c}}$ki institut SANU, Belgrade, 2004.

\bibitem{P11}
M. Pavlovi\'{c}, A Schwarz lemma for the modulus of a vector-valued
analytic functions, \textit{Proc. Amer. Math. Soc.}, \textbf{139}
(2011), 969--973.

\bibitem{Ra}  W. Ramey,
Local boundary behavior of pluriharmonic functions along curves,
\textit{Am.  J. Math.,} {\bf
108} (1986), 175--191.


\bibitem{RU}  W. Ramey and D. Ullrich,
The pointwise Fatou theorem and its converse for positive pluriharmonic functions,
\textit{Duke Math. J.,} {\bf
49} (1982), 655--675.

\bibitem{Ro}  B. Rodin,
Schwarz's lemma for circle packings, \textit{Inven. Math.,} {\bf
89} (1987), 271--289.


\bibitem{Ru2} W. Rudin,
\textit{Function theory in the unit ball of $\mathbb{C}^{n}$,}
Springer-Verlag, New York, Heidelberg, Berlin, 1980.

\bibitem{Ta} S. Takahashi, Univalent functions in several complex variables, \textit{Ann. Math.}, {\bf 53} (1951), 464--471.

\bibitem{T}  H. Tsuji,
A generalization of Schwarz lemma, \textit{Math. Ann.,} {\bf
256} (1981), 387--390.


\bibitem{U-1938} H. Unkelbach, \"Uber die Randverzerrung bei konformer Abbildung,
\textit{Math. Z.,} {\bf
43} (1938), 739--742.


\bibitem{Va}  J. V\"ais\"al\"a,
{\it Lectures on n-dimensional quasiconformal mappings,} Lecture
Notes in Mathematics, {\bf 229} (Springer, Berlin, 1971), 156pp.


\bibitem{Vl}  V. S.  Vladimirov,
\textit{Methods of the theory of functions of several complex
variables,} (in Russian), M. I. T. Press, Cambridge, Mass., 1966.

\bibitem{Vu}  M. Vuorinen,
{\it Conformal Geometry and Quasiregular Mappings,} Lecture Notes in Mathematics, {\bf 1319} (Springer, Berlin, 1988), 209pp.


\bibitem{Wu}  H. Wu,
Normal families of holomorphic mappings, \textit{Acta Math.,} {\bf
119} (1967), 193--233.

\bibitem{X18}
Z. Xu, A Schwarz-Pick lemma for the norms of holomorphic mappings in Banach spaces, \textit{Complex Var. Elliptic Equ.}, \textbf{63}  (2018),  1459--1467.





\bibitem{Y}  S. T. Yau,
A general Schwarz lemma for K\"ahler manifolds, \textit{Amer. J.
Math.,} {\bf 100} (1978), 197--203.


\bibitem{Z18}
J.-F. Zhu,
Schwarz lemma and boundary Schwarz lemma for pluriharmonic mappings, \textit{Filomat,} \textbf{32} (2018), 5385--5402.


\bibitem{Z19}
J.-F. Zhu, Schwarz-Pick type estimates for gradients of
pluriharmonic mappings of the unit ball, \textit{Results Math.},
\textbf{74} (2019),  74--114.


\end{thebibliography}
\end{document}